\documentclass[10.9pt]{amsart}
\usepackage{amsmath}
 \usepackage{amscd}
 \usepackage{amssymb}
\usepackage[applemac]{inputenc}
\usepackage{tikz-cd}
\usepackage{ulem}
\usepackage[applemac]{inputenc}
\usepackage[all]{xy}
\usepackage{color}

\usepackage{xparse}

\NewDocumentCommand{\dslash}{s}{%
  \IfBooleanTF{#1}
    {\big/\mkern-7mu\big/}
    {/\mkern-6mu/}%
}

\newtheorem{definition}{Definition}[section]

\newtheorem{proposition}{Proposition}[section]

\newtheorem{theorem}[proposition]{Theorem}

\newtheorem{remark}{Remark}[section]

\DeclareMathOperator{\PGL}{PGL}
\DeclareMathOperator{\Ker}{Ker}

\DeclareMathOperator{\Sing}{Sing}
\DeclareMathOperator{\Aut}{Aut}

\DeclareMathOperator{\Pic}{Pic}

\DeclareMathOperator{\rank}{rank}

\title{ \small The Igusa quartic and the Prym map, \\ with some rational moduli  }
\author[  A. Verra]{Alessandro Verra}
 \address{Dipartimento di Matematica \\ %
Universit\'a degli Studi di Roma TRE \\ %
Largo San Leonardo Murialdo \\ %
00146 Roma \\ %
Italy} \thanks{This work was partially supported by INdAM-GNSAGA and by the projects PRIN-2015 'Geometry of Algebraic Varieties'
and PRIN-2017 'Moduli Theory and Birational Classification'}
 \email{verra@mat.uniroma3.it}
 
\begin{document}
    \maketitle 
    
   \begin{abstract}  In this paper the ubiquity of the Igusa quartic $B \subset \mathbb P^4$ shows up again, this time related to 
the Prym map $  \mathfrak p : \mathcal R_6 \to \mathcal A_5$. We introduce the moduli space $\mathcal X$ of those quartic threefolds $X$ cutting twice a quadratic section of $B$. A general $X$ is $30$-nodal and the intermediate Jacobian $J(X)$ of its natural desingularization is a $5$-dimensional p.p. abelian variety.  Let $\frak j: \mathcal X \to \mathcal A_5$ be the period map sending $X$ to $J(X)$, in the paper we study $\frak j$ and its relation to $\frak p$. As is well known the degree of $\frak p$ is $27$ and its monodromy group endows any smooth fibre $F$ of $\frak p$ with the incidence configuration of $27$ lines of a cubic surface. Then the same monodromy  defines a map $ \mathfrak j':  \mathcal D_6  \to \mathcal A_5$ of degree $36$, with fibre the configuration of $36$ 'double-six' sets of lines of a cubic surface.  We prove that $\frak j = \frak j' \circ \phi$, where $\phi: \mathcal X \to \mathcal D_6$ is birational. This provides a geometric description of $\frak j'$.  Finally we describe the relations between the different moduli spaces
considered and prove that some, including $\mathcal X$, are rational.
  
 \end{abstract} 
\section{Introduction and preliminaries}
The Igusa quartic is a well known quartic threefold  in  the complex projective space $\mathbb P^4$. It originates from classical  Algebraic Geometry  and Invariant Theory, see \cite{B}  chapter  V. In more recent times it was frequently reconsidered, starting from some work of Igusa and of van der Geer, \cite{I, VdG};  we will denote it by $B$.   Its dual $B^* \subset \mathbb P^{4*}$ is the equally  famous \it Segre cubic primal\rm, \cite{S1}. $B^*$ is the unique invariant cubic threefold  under the action of the symmetric group $\frak S_6$, determined by its standard irreducible representation of dimension 5, \cite{FH} 4.  Its isomorphic dual acts on $B$ defining a natural isomorphism $\frak S_6 \cong \Aut B$. Finally let $\mathbb P^4 \subset \mathbb P^5$ be the hyperplane
$\lbrace z_1+ \dots + z_6 = 0 \rbrace$, then the equation of $B$ is obtained putting $t = \frac 14$ in the pencil  
 $$ z_1^4 + ... + z_6^4 - t(z_1^2+ ... + z_6^2)^2 = 0, $$
of all $\frak S_6$-invariant quartics of $\mathbb P^4$, \cite{B3, CKS, VdG}.  The most beautiful and deep roots of Algebraic Geometry, since classical times to now, are strongly represented by varieties like $B$ and $B^*$. Therefore we will not give up on providing  a few more remarks about their geometry and history. It is also due to remind of the ubiquity of $B$ in Algebraic Geometry, so that $B$ is marking beautifully the landscape at several spots.  Let $\mathcal R_6$ be the moduli space of genus $6$ Prym curves and $\mathcal A_5$ that of principally polarized abelian varieties of dimension $5$, in our paper we recover this ubiquity at a new place, addressing the links of $B$  to the Prym map
 $$   \mathfrak p : \mathcal R_6 \to \mathcal A_5. $$    
  We show that the beautiful geometry of $\mathfrak p$, revealed in \cite{DS} and \cite{Do}, nicely relates with the geometry of $B$. We deduce by the way the rationality of most of the moduli spaces involved.   Our work frequently relies on the  papers \cite{CKS} and \cite{FV}. \par Theorems {\sf A, B, C, D} in this section summarize our results.   To state these we cannot avoid some preliminaries.  We begin from  a quintic Del Pezzo surface $S$ and its realization  
\begin{equation} \label{UNO}
S = \mathbb P^5 \cap \mathbb G \subset  \mathbb P^9
\end{equation}
as a smooth linear section of the Pl\"ucker embedding $\mathbb G$ of the Grassmannian of planes of $\mathbb P^4$. As is well known, all these surfaces are biregular and projectively equivalent as linear sections of $\mathbb G$.  From now on $\mathbb P$ will denote the \it universal plane over $S$ \rm that is
\begin{equation}
\mathbb P = \lbrace (x, y) \in \mathbb P^4 \times S  \ \vert \ x \in \mathsf P_y \rbrace,
\end{equation}
where $\mathsf P_y \subset \mathbb P^4$ is the plane whose parameter point is $y$. Consider the natural projections 
\begin{equation} \label{DUE}
\begin{CD}
{\mathbb P^4} @<t<< {\mathbb P}@>u>> S, \\
\end{CD}
\end{equation}
then $u$ is a $\mathbb P^2$-bundle structure on $\mathbb P$ and $t$ is defined by the tautological sheaf $\mathcal O_{\mathbb P}(1) := t^* \mathcal O_{\mathbb P^4}(1)$. We notice that $t$ has degree two and consider its Stein factorization:
\begin{equation} \label{TRE}
\begin{tikzcd}
{} & {\overline{\mathbb P}} \arrow{dr}{\overline t} \\
{\mathbb P} \arrow{ur}{c} \arrow[rr]{}{t} && {\mathbb P^4.}
\end{tikzcd}
\end{equation}
 It is known that the branch divisor of $\overline t$ is $B$ and that $c$ is a small resolution of $\overline{\mathbb P}$,  \cite{CKS} 2.32. A very important element of our picture is represented  by  the linear system of divisors $$ \vert \mathcal O_{\mathbb P}(2) \vert. $$ 
This is studied in \cite{FV} 2, we recall that a general $Q \in \vert \mathcal O_{\mathbb P}(2) \vert$ is smooth and that $u \vert Q: Q \to S$ is a conic bundle.  Its discriminant is a smooth Prym curve $(C, \eta)$  of genus $6$, that is $C$ is a smooth, integral genus $6$ curve and $\eta \in \Pic^0 C$ is a non trivial $2$-torsion element.  Hence $(C,\eta)$ defines a point of the moduli space $\mathcal R_6$. Actually $C$ is canonically embedded in $S \subset \mathbb P^5$ and $(C,\eta)$ has general moduli.  In particular we meet here the space of $\Aut \mathbb P$-isomorphism classes of the conic bundles $u \vert Q: Q \to S$. We define such a space as a GIT quotient, fixing for it the notation: \begin{equation} \vert \mathcal O_{\mathbb P}(2) \vert \dslash \Aut \mathbb P =: \mathcal R^{cb}_6. \end{equation}
  Let $Q$ be general, we show in section 6 that the datum of $u \vert Q$ is equivalent, modulo $\Aut \mathbb P$,  to the data of the triple $(C, \eta,s)$, where $(C, \eta)$ is the discriminant of $u \vert Q$ and $s: C \to \mathbb P$ is the map sending $x$ to $\Sing u\vert Q^*(x)$. We define $(C,\eta,s)$ as the \it Steiner map of $u \vert Q$.   \rm Therefore $\mathcal R^{cb}_6$ is  the space of $\Aut \mathbb P$-isomorphism classes of Steiner maps as well.     Notice that $\Aut S$ is isomorphic to $\Aut \mathbb P$ via its action on the fibres of $u$. Moreover $\Aut S$ is the symmetric group $\frak S_5$ and acts linearly on $\mathbb P^4$ and $\mathbb P^5$ inducing $\Aut \mathbb P$ and $\Aut S$, \cite{SB}. Let $\sigma: S \to \mathbb P^2$ be the contraction of four disjoint lines of $S$ and $\Aut_{\sigma}S \subset \Aut S$  the group leaving invariant the exceptional divisor $E$ of $\sigma$, then $\Aut_{\sigma} S$ is $\frak S_4$ and acts on $\sigma(E)$ by permutations. Now let $(C, \eta, s)$ be a Steiner map then $\sigma \vert C$ defines the line bundle $M := \sigma \vert C^*\mathcal O_{\mathbb P^2}(1)$, that is a sextic model of $C$ in $\mathbb P^2$. In the paper we will also consider the GIT-quotient of $\Aut \mathbb P$-isomorphism classes of $4$-tuples  $(C, \eta, s, M)$:
\begin{equation}
  \vert \mathcal O_{\mathbb P}(2) \vert \dslash \Aut_{\sigma} S =: \widetilde{\mathcal R}_6^{cb}.
\end{equation}
  To continue let us introduce one more actor  in our story: the moduli space of $4$-nodal Prym plane sextics $(C',\eta')$. Here $C'$ is a $4$-nodal plane sextic and $\eta' \in \Pic^0 C'$ is a non zero $2$-torsion element. We will denote such a moduli space as $ \mathcal R^{ps}_6$.   In the paper we prove that:
\medskip \par
{\sf Theorem A}  \it The moduli space $\mathcal R^{ps}_6$ of $4$-nodal Prym plane sextics is rational. \rm
\medskip \par \noindent
 Indeed we prove the rationality of ${    \widetilde{\mathcal R}^{cb}_6  }$ and the birationality of $ \widetilde{\mathcal R}^{cb}_6  $ and $\mathcal R^{ps}_6$, see  theorems (\ref{NEWRATIONAL}) and  (\ref{BIRATIONAL}). The latter follows from a construction and the methods in \cite{FV} 2. In short we rely on the next commutative diagram, to be explained in details in sections 4 and 6: 
$$
\begin{tikzcd}
{} {\mathbb P^2 \times \mathbb P^2} \arrow[dr,dashed]{}{t'} \arrow[rr,dashed]{}{\epsilon} && {\mathbb P}\arrow[ld]{t}{t} \arrow[rr]{}{i} && {\mathbb P^{15}} \\
  & {\mathbb P^4.}  
\end{tikzcd}
$$
Here the map $t'$ has degree $2$ and is defined  by $\vert \mathcal I_{\mathsf e}(1,1) \vert$, where $\mathcal I_{\mathsf e}$ is the ideal sheaf of $\mathsf e$ in $\mathbb P^2 \times \mathbb P^2$ 
and $\mathsf e$ is a set of $4$ general points. The map $\epsilon$ is birational and $\vert \mathcal I^2_{\mathsf e}(2,2) \vert$ defines $\epsilon \circ i$,  where $i$ is
the embedding defined by $\mathcal O_{\mathbb P}(2)$.  Notice also that $\epsilon$ defines by pull-back a linear isomorphism  $\epsilon':  \vert \mathcal O_{\mathbb P}(2) \vert  \to \vert \mathcal I^2_{\mathsf e}(2,2) \vert$. Using this we will prove the birationality of $ \widetilde{\mathcal R}^{cb}_6  $ and $\mathcal R^{ps}_6$.  Then, focusing our view on $B$ and $\mathbb P^4$, we will see very interesting families of threefolds and geometric configurations. These relate  $B$ and the Prym map, let us summarize our results about that. \par
At first, starting from a general $Q \in \vert \mathcal O_{\mathbb P}(2) \vert$, we list some geometric objects related to it and define their moduli spaces. Let $\iota: \mathbb P \to \mathbb P$ be the birational involution induced  by $t$ and let $\overline Q := \iota^*Q$, we notice that $\overline Q \in \vert \mathcal O_{\mathbb P}(2) \vert$ and that the same is true for the ramification divisor of $t$
$$R := t^{-1}(B) \in \vert \mathcal O_{\mathbb P}(2) \vert. $$ 
In particular the set of the fixed points of the involution $\iota^*$ is the union
$
t^* \vert \mathcal O_{\mathbb P^4}(2) \vert \cup \lbrace R \rbrace.
$ \par
A first  object associated to $Q$ is $X := t_*Q = t_* \overline Q$. $X$ is a quartic threefold with $30$ nodes, since $\Sing X = X \cap \Sing B$. Let $\tilde X$ be $X$ blown up
at $\Sing X$. Then its intermediate Jacobian is the Prym variety $P(C, \eta)$ of the discriminant of the conic bundle $u: Q \to S$. Since $t^*X = Q + \overline Q$ is split and $B$ is the branch locus of $t$, we have $B \cdot X = 2A$ where $A$ is a quadric. Hence $X$ defines a pencil of quartic threefolds with $30$ nodes
$$
\lambda a^2 + \mu b = 0,
$$
where $B = {\rm div}(b)$ and $A = {\rm div}(a)$. We define a pencil $P$, generated by $b$ and some $a^2$, as an \it Igusa pencil. \rm Let $X \in P$ we also say that $X$ is an \it $E_6$-quartic threefold. \rm As we will see, $X$ is related to the root lattice $\mathsf E_6$ and its corresponding simple Lie Algebra.  So the ubiquity of $B$ actually reflects $\mathsf E_6$ and its ubiquity. In $\vert \mathcal O_{\mathbb P^4}(4) \vert$  the family of these threefolds is a cone $\mathbb V$  with vertex $B$ over the family $\mathbb V_2$ of
all double quadrics $2A$ and $\mathbb V_2$ parametrizes  Igusa pencils. Again, the projective isomorphism classes of $\mathbb V$ and $\mathbb V_2$ can be constructed as GIT quotients: 
\begin{equation} \label{TWOQUOTIENTS}
 \mathcal X  =: \mathbb V \dslash \Aut B \ \text{ and}\ \mathcal P^I =: \mathbb V_2 \dslash \Aut B.
\end{equation}
Here $\Aut S \subset \Aut B$ is the stabilizer of $S$ and $\Aut B \cong \frak S_6$, see section 7.  Clearly one has $t^*X = Q + \overline Q$ and $Q, \overline Q$ generate the pencil $t^*P$. Now consider the plane $\mathsf P_y = t_*u^*(y)$, $y \in S$.  Restricting $X$ to it we have a union of two conics: $ X \cdot \mathsf P_y = t_*(Q \cdot \mathbb P_y) \cup t_*(\overline Q \cdot \mathbb P_y)$.
Therefore $X$ has  two family of conics,  induced from the conic bundle structures $u \vert Q$ and $u \vert \overline Q$ over $S$. \par Notice that $\mathsf P_y \cdot B$ is a double conic. This follows since $B$ is the focal locus of the congruence of planes $S$, that is the branch divisor of $t$. See section 2 and  \cite{CS}  about the classical \it theory of foci \rm for a family of linear spaces in $\mathbb P^n$. Now the intriguing feature of $B$ shows up as follows. \par $B$ is actually the focal locus for six congruences of planes which are smooth linear sections of $\mathbb G$. The set of these surfaces is the orbit of $S$ under the action of $\Aut B$ on $\mathbb G$. Indeed $\Aut B$ is the symmetric group $\frak S_6$ and the stabiizer of $S$ under its action is $\Aut S \cong \frak S_5$.  We will say more about this classical fact in section $3$, [\footnote{\tiny Passing to the Segre primal $B^* \subset \mathbb P^{4*}$ and the dual Grassmannian of lines, these surfaces define the six quintic Del Pezzo components of the well known  Fano surface of $B^*$, cf. \cite{D3} 2. }].  Then  $X$ is endowed with six pairs of conic bundle structures over the six surfaces: a \it double six \rm so to say. Let  $(X,Q)$ be a pair such that $Q \subset t^*X$ and $Q \in \vert \mathcal O_{\mathbb P}(2) \vert$,  we will say that $(X,Q)$  is a \it marked $E_6$-quartic threefold. \rm Finally let  us set
\begin{equation} \widetilde{ \mathcal X } :=  \lbrace (X,Q) \ \vert \  \text{\it $X$ is a marked $E_6$-quartic threefold}, \  X \in \mathbb V \  \rbrace \dslash \Aut B, \end{equation}
then the GIT quotient  $\widetilde{ \mathcal X }$ is endowed with the degree two forgetful map $q: \widetilde{ \mathcal X } \to  \mathcal X $, whose associated involution is induced by the  exchange of $Q$ and $\overline Q$. Since $\Aut S \cong \frak S_5$, the inclusion $\Aut S \subset \Aut B$ induces a degree six morphism $d: \mathcal R^{cb}_6 \to \widetilde{ \mathcal X }$ and the diagram
\begin{equation} \label{TRIANGLEDIAGRAM}
 \begin{tikzcd}[column sep=7pc] 
          \\
  {\mathcal R^{cb}_6}\arrow{rd}{d}   \arrow{r}{u := q \circ d}  &
  { \mathcal X } &
   \\
  & {\widetilde{ \mathcal X }.} \arrow{u}{q} \\
\end{tikzcd}
 \end{equation}
 Fixing the contraction $\sigma: S \to \mathbb P^2$, that is $\Aut_{\sigma} \subset \Aut S$, we define in the same way
$$   \widetilde{\mathcal R}^{cb}_6  := \vert \mathcal O_{\mathbb P}(2) \vert \dslash \Aut_{\sigma} S. $$
This parametrizes $\Aut \mathbb P$-isomorphism classes  of triples $(X,Q, \sigma')$, where $\sigma': S \to\mathbb P^2$ is one of the five contractions of $S$ to $\mathbb P^2$.  The relations between the quotient spaces from above are described in the forthcoming sections and summed up in the next theorem and diagram.   The connections to the Prym map and the $27$ lines of the cubic surface are discussed in section 7.  \par 
  Let $\tilde d: \widetilde {\mathcal R}^{cb}_6 \to \mathcal R^{cb}_6$ and $d: \mathcal R^{cb}_6 \to  \mathcal X $ be the forgetful maps. Clearly these are induced by the inclusions $\Aut_{\sigma} S \subset  \Aut S \subset \Aut B$, so that $\deg \tilde d = 5$, $\deg d = 6$. Now let us define the following maps: $p: \mathcal R^{ps}_6 \to \mathcal R_6$ is induced by the assignment $(C', \eta') \to (C,\eta)$,  where $\eta \cong \nu^*\eta'$ and $\nu: C \to C'$ is the normalization. $\alpha:   \widetilde{\mathcal R}^{cb}_6   \to \mathcal R^{ps}_6$ is the \it birational \rm map from the proof of theorem (\ref{BIRATIONAL}). Consider also the forgetful map  $n: {\mathcal R}^{cb}_6 \to \mathcal R_6$ 
and finally the \it period map \rm  $$ \mathfrak j :  \mathcal X  \to \mathcal A_5, $$ induced by the assignment of $X$ to $J\tilde X$. Let $  \mathfrak p: \mathcal R_6 \to \mathcal A_5 $ be \it Prym map \rm then we have:
  
  \medskip \par
{\sf Theorem B}  \it The next diagram is commutative. Moreover the period map $ \mathfrak j :  \mathcal X  \to \mathcal A_5$  has degree $36$ and the forgetful map $n: \mathcal R^{cb}_6 \to  \mathcal X $ has degree $16$.  
\begin{equation} \label{MAINDIAGRAMMA}
\begin{tikzcd}[column sep=9pc] 
  {  \widetilde{\mathcal R}^{cb}_6  } \arrow{d}{\tilde d}   \arrow{r}{\alpha} &{\mathcal R^{ps}_6} \arrow{r} {p}& {\mathcal R_6}   \arrow{d}{  \mathfrak p } \\
  {\mathcal R^{cb}_6}\arrow{rd}{d}  \arrow{rru}{n} \arrow{r}{u := q \circ d}  &
  { \mathcal X } \arrow{r}{ \mathfrak j } &
  {\mathcal A_5} \\
  & {\widetilde{ \mathcal X }} \arrow{u}{q} \\
\end{tikzcd}
\end{equation}
 \medskip \par
 {\sf Theorem C}  \it  The moduli spaces of $E_6$-quartic threefolds and of Igusa pencils are rational.  
\rm
\medskip \par
   See (\ref{RATIGUSA}),  (\ref {THEOREM f}). Let $\mathbb F$ be the configuration of $27$  lines of a smooth cubic surface, as we know the fibre of the Prym map $\frak p$ is reflected by $\mathbb F$.
   Now,  betting that the fibre of the period map $ \mathfrak j :  \mathcal X  \to \mathcal A_5$  reflects the configuration of double sixes of $\mathbb F$ is just natural. The numbers $\deg  \mathfrak j  = 36$ and $\deg n = 16$ suggest indeed that: $36$ is the number of double sixes in $\mathbb F$ and $16$ the number of double sixes containing an element of $\mathbb F$. A celebrated theorem of Donagi shows that the monodromy group of $ \mathfrak p $ is the Weyl group $W(E_6)$ of permutations preserving the incidence relation of $\mathbb F$, \cite{Do} 4.2.    
Let $\mathsf a \in \mathcal A_5$ be general and $  \mathfrak p ^{-1}(\mathsf a) = \mathbb F_a$, the monodromy of $  \mathfrak p $ allows us to think of $\mathbb F_{\mathsf a}$ as $\mathbb F$. Denoting by $\rm DS_{\mathsf a}(\mathbb F)$ its set of double sixes let us consider:
\begin{enumerate}
\item $ \mathcal D_6  := \lbrace (\mathsf s, \mathsf a) \vert \ \mathsf s \in {\rm DS_{\mathsf a}(\mathbb F), \ \mathsf a \in \mathcal A_5}  \rbrace$ and its degree $36$ forgetful map
$$ \mathfrak j ':  \mathcal D_6  \to \mathcal A_5. $$ 
\item $\mathcal R' := \lbrace (\mathsf s, \mathsf a, \mathsf l) \ \vert \ (\mathsf s, \mathsf a) \in  \mathcal D_6 , \ \mathsf l \in \mathsf s \subset \mathbb F_{\mathsf a} \rbrace$ and its degree $16$ and $12$  forgetful maps
$$n': \mathcal R' \to \mathcal R_6 \ \text{and} \ u': \mathcal R' \to  \mathcal D_6 .$$
  \end{enumerate}
   The irreducibility of $ \mathcal D_6 $, and of $\mathcal R'$, follows by monodromy. Indeed the monodromy group of $  \mathfrak p $ is $W(E_6)$ and acts transitively on the configuration of double sixes of $\mathbb F$, \cite{D1} 9.1.4.  Now the above rational maps and theorem B define the commutative diagrams 

 \begin{equation} \label{DIAGRAMMATWIN}
\begin{CD}
{\mathcal R'} @>{n'}>> {\mathcal R_6} \ \ \ \ \ @. @. {\mathcal R^{cb}_6} @>{n}>> {\mathcal R_6} \\
@VV{u'}V @VV{  \mathfrak p }V \ \ \ \ \ \ @.@VV{u'}V @VV{  \mathfrak p }V \\
{ \mathcal D_6 }@>{ \mathfrak j'}>>{\mathcal A_5} \ \ \ \ \ \  @. @. { \mathcal X }@>{ \mathfrak j }>>{\mathcal A_5}, \\
\end{CD}
 \end{equation}
so that we can give the following definitions.
 
\begin{definition} $  \mathfrak j ':  \mathcal D_6  \to \mathcal A_5$ is the universal family of double sixes of $  \mathfrak p $. \end{definition} 
\begin{definition} $n': \mathcal R' \to \mathcal R_6$ is the universal pointed family of double sixes of $  \mathfrak p $. \end{definition}
  
In section 8, relying on Donagi's tetragonal construction, we conclude as follows.
  \medskip \par
{\sf Theorem D}  \it  The period map $ \mathfrak j :  \mathcal X  \to \mathcal A_5$ of $E_6$-quartic threefolds and the universal family  $ \mathfrak j ':  \mathcal D_6  \to \mathcal A_5$, of double sixes of $\frak p$,  are birationally equivalent over $\mathcal A_5$.  \rm
\medskip \par
The birationality of  ${\mathcal R}'$ and $\mathcal R^{cb}_6$ also follows.  So: \it this is the geometry of  $E_6$-quartic threefolds and of their periods,   relating their double sixes of conic bundles to the classical ones. \rm
   \medskip \par
{${\mathsf {Aknowledgements}}$} \par We are indebted to Gabi Farkas for several conversations during the preparation of the related joint paper \cite{FV}. We also thank Ivan Cheltsov, Igor Dolgachev 
and Alexander Kuznetsov for a very useful correspondence on the subject of this  paper.  
 
\section{Plane geometry in $\mathbb P^4$: the congruence $S$ and its focal locus $B$}
In this section we introduce more geometry of the quintic Del Pezzo surface $S$ and the Igusa quartic $B$. On the occasion, before describing our results,  we revisit some classical notions related to Grassmannians and their history.  About this a warning is due. It is known that an unexpected number of decades passed, after
the outstanding work of Grassmann,
 in order  to fix  language and theory on Grassmannians, see e.g.  the biographical paper \cite{D} by Dieudonn\`e. This implies that the words in use and their meaning  have undergone quite contrasting variations.  \par
Let $\mathbb G$ be the Pl\"ucker embedding of the Grassmannian of $r-$spaces in $\mathbb P^n$ with $0 < r < n-1$.  Possibly adopting the classical language, we will say that a surface $Y \subset \mathbb G$ is a \it congruence \rm of $r$-spaces of $\mathbb P^n$.  In this language the \it order \rm of the congruence $Y$ and its \it class \rm are well defined as follows. For $y \in \mathbb G$ let $\mathsf P_y \subset \mathbb P^n$ be the corresponding $r$-space, consider the universal $r$-space \begin{equation}  \label{PLANE}
\mathbb I := \lbrace (x,y) \in \mathbb P^n \times \mathbb G \ \vert \ x \in \mathsf P_y \rbrace,
\end{equation}
 and its projections
\begin{equation}
\begin{CD}
{\mathbb P^n} @<{\mathsf t}< <{\mathbb I} @>{\mathsf u}>> {\mathbb G}. \\
\end{CD}
\end{equation}
Then $\mathsf u$ is a $\mathbb P^r$-bundle structure on $\mathbb I$ and $\mathsf t$ is its natural tautological map. In the Chow ring $CH^*(\mathbb I)$ let $h$ be the pull-back by $\mathsf t$ of the hyperplane class, then the degree of the $0$-cycle class 
\begin{equation}
\mathsf u_* (h^{r+2} \mathsf u^*[Y]) = (\mathsf u_* h^{r+2})[Y]\in CH^*(\mathbb G)
\end{equation}
is the \it order \rm of $Y$. Counting multiplicities this is just the \it degree of the scroll \rm union of the $r$-spaces parametrized by $Y$.The \it class \rm of $Y$ is the number of its elements 
not intersecting properly a general space of codimension $r$.  Let $\mathcal V_Y$ be the restriction to $Y$ of the locally free sheaf $\mathcal V = \mathsf u_* \mathsf t^* \mathcal O_{\mathbf P^n}(1)$,  in modern terms we can define the order $a$ and the class $b$ of $Y$ via Chern classes:
\begin{equation}
a = \deg (c_1^2(\mathcal V) - c_2(\mathcal V))  \ , \ b = \deg c_2(\mathcal V).
\end{equation}
Then $a$ is the degree of the cycle $\mathsf t_* \mathsf u^* Y$, that is  the degree of the second Segre class $s_2(\mathcal V)$. The degree of $Y$ in the Pl\"ucker space is $a + b = c_1^2(\mathcal V)$.  Let $\alpha, \beta \in CH_2(\mathbb G)$ be the two classes of the families of planes contained in $\mathbb G$,  we denote by $\alpha$ the class of a plane parametrizing the $r$-spaces through a fixed codimension $3$ space of $\mathbb P^n$. By $\beta$ we denote the class of a net of $r$-spaces in a fixed $(r+1)$-space of $\mathbb P^n$. Notice that $CH_2(\mathbb G) = \mathbb Z \alpha \oplus \mathbb Z \beta$ and that
$$
[Y] = a \alpha + b \beta.
$$
  \begin{remark} \rm In spite of the generality of Grassmann's foundations, most of geometric research was  prevalently focusing on the Grassmannian of lines of $\mathbb P^3$ for long time, \cite{K} XXII 4.
For decades the tripartition of names \it ruled surface\rm, \it congruence\rm, \it complex \rm was largely in use for families of lines in $\mathbb P^3$ of dimensions $1$, $2$, $3$, \cite{Lo} 7.
Passing to any $\mathbb G$, the word \it complex \rm evolved as synonimous of hypersurface in $\mathbb G$. 
Among many exceptions to this trend, let us mention two papers on the Grassmannians of $\mathbb P^4$ by Segre and Castelnuovo,  \cite{S2, C}. Both are relevant in what follows.
\end{remark}
For simplicity let  $Y \subset \mathbb G$ be a smooth integral variety and let us consider the morphism
\begin{equation}
\mathsf t \vert \mathsf u^*Y:\mathsf u^*Y \to  \mathbb P^n,
\end{equation}
defined by the sheaf $\mathsf t^* \mathcal O_{\mathbb P^n}(1)$. Often we will say that the $\mathbb P^r$-bundle $\mathsf u^*Y$ is \it the universal $r$-space over $Y$ \rm and that
$\mathsf t \vert \mathsf u^* Y$ is \it  its tautological map. \rm  The classical \it Theory of Foci \rm is concerned with the ramification of this map and the infinitesimal deformations in the family $Y$. Let $y \in Y$, following this theory we say that the \it focal locus  of the $r$-space \rm $\mathsf P_y$ is the restriction to $\mathsf P_y$ of the ramification scheme
of $\mathsf t \vert Y$. See \cite{CS} and  \cite{Se} 4.6.7 for a modern reconstruction. \par The general case of interest is clearly when \it $\mathsf t \vert \mathsf u^*Y$ is a generically finite morphism. \rm So
we assume this since now and we will have in particular $\dim Y = n - r$. Then we fix the notation
\begin{equation}
B_Y \subset \mathbb P^n
\end{equation}
for the branch scheme  of $\mathsf t \vert \mathsf u^*Y:\mathsf u^*Y \to \mathbb P^n$ and say that \it $B_Y$ is the focal locus \rm of $Y$. If $\deg \mathsf t \vert \mathsf u^*Y \geq 2$ then $B_Y$ is a hypersurface, we call it \it the focal hypersurface of $Y$. \rm  Finally we will say, with the usual modern language, that  the \it fundamental locus \rm of the family $Y$ is \begin{equation}
F := \lbrace x \in \mathbb P^n \ \vert \ \dim (\mathsf t \vert \mathsf u^*Y)^{-1}(x) \geq 1 \rbrace. \end{equation}
In classical terms $F$ is indeed the locus of points in $\mathbb P^n$ which are singular for $Y$: quite conflicting with modern language.
Since now we fix $n = 4$ and $r = 2$ so that $Y$ is a congruence of planes in $\mathbb P^4$. Moreover we assume that $Y$ is a \it smooth linear section of $\mathbb G$. \rm and keep for this surface the notation $S$ adopted in the introduction.  In particular we have as in (\ref{UNO})
$$
S = \mathbb P^5 \cap \mathbb G \subset \mathbb P^9.
$$
$S$ is therefore a quintic Del Pezzo surface, see \cite{D} 8.5  for a detailed description. Let us also recall that the group $\Aut S$ is isomorphic to the symmetric group $\frak S_5$, cf. \cite{SB}. We have
\begin{equation}
\Aut S \subset \Aut \mathbb P^4
\end{equation}
and, up to conjugation, its action on $\mathbb P^4$ is the permutation of coordinates. As above let $\mathcal V = \mathsf u_* \mathsf t^* \mathcal O_{\mathbf P^n}(1)$ be the rank $3$ vector bundle whose projectivization is the universal plane over $S$, then the following properties
of $\mathcal V$ are standard and well known:
  \begin{enumerate}  \it
 \item   $c_1(\mathcal V) \cong \mathcal O_S(1)$, 
 \item $c_1(\mathcal V)^2 - c_2(\mathcal V) = 2$,
\item $c_2(\mathcal V) = 3$
 \item $h^0(\mathcal V) = 5$,
 \item $h^i(\mathcal V) = 0$ for $i \geq 1$.
 \end{enumerate}
In particular the congruence of planes $S$ has order $2$ and class $3$ in $CH_2(\mathbb G)$ that is
\begin{equation}
[S] = 2\alpha + 3\beta.
\end{equation}
To see that $S$ has order $2$ just consider a general point $x \in \mathbb P^4$ and the Schubert variety
$$ \mathbb G_x := \lbrace (x,y) \vert x \in \mathsf P_y \rbrace = \mathsf t^*(x). $$
Then $\mathbb G_x$ is a smooth quadric of dimension $4$ and the linear span $\mathbb P^5$ of $S$ cuts on it a linear section. Since $x$ is general we can assume that
$\mathbb G_x$ and $\mathbb P^5$ intersects transversally. Hence the order of $S$ is $2 = \deg \mathbb G_x$ and its class is $3 = \deg c_2(\mathcal V)$. $\mathcal V$ fits in the standard exact sequence
$$
0 \to \mathcal U \to H^0(\mathcal O_S(1)) \otimes \mathcal O_S \to\mathcal V \to 0.
$$
Passing to the long exact sequence, it is easy to deduce that $h^0(\mathcal V) = 5$ and $h^i(\mathcal V) = 0$ for $i \geq 1$. Now let $G(5,9)$ be the Grassmannian of $5$-spaces in the Pl\"ucker space of $\mathbb G$, as is well known the action of $\Aut \mathbb P^4$ on it has a unique open orbit, which is the family of $5$-spaces transversal to $\mathbb G$. This implies that $\mathcal V$ is the unique rank $3$ vector bundle on $S$ satisfying the above properties and defining an embedding of $S$ in $\mathbb G$. In particular $S$ is equivariant with respect to $\Aut S$.
\begin{remark}  \rm $\mathcal V$ is also known as  \it the rank $3$ Mukai bundle of $S$, \rm since it is related to Mukai's higher Brill-Noether theory for a general genus $6$ canonical curve $C$.
Indeed $C$ is embedded in $S$ and $\mathcal V \otimes \mathcal O_C$ is the kernel of the evaluation map of the unique stable vector bundle on $C$, of rank $2$ and canonical determinant, with  $5$ linearly independent global sections, \cite{M, FV}. 
\end{remark}
Now we want to study the universal plane over $S$ and its tautological morphism. To this purpose we fix the notation $\mathbb P$ for such a universal plane and consider the diagram
\begin{equation} \label{Stein} 
\begin{tikzcd}
 {} &  {\mathbb I}  \arrow{dl}{\mathsf t \vert \mathbb P}   \arrow{dr} {\mathsf u \vert \mathbb P} \\
 {} \mathbb P^4 & \arrow{l}{t} \mathbb P \arrow{u}{} \arrow{r}{u}  & S
 \end{tikzcd}
\end{equation}
 where $t: \mathbb P \to \mathbb P^4$ is the  tautological morphism, $u: \mathbb P \to S$ is the universal plane over $S$ and the vertical arrow is the inclusion.
Since the  order of $S$  is  two then $t$ is a morphism of degree two.  \par
 Our main interest is in \it the focal locus of the congruence $S$, \rm that is the branch hypersurface of $t$. We will simply  denote it by $B$.  The ramification divisor of $t$ will be denoted as
\begin{equation}
R \subset \mathbb P.
\end{equation}
\begin{proposition} $B$ is a quartic threefold and $R$ is an element of $\vert \mathcal O_{\mathbb P}(2) \vert$. \end{proposition}
\begin{proof} Let $K$ be a canonical divisor of $\mathbb P$ and $H \in \vert \mathcal O_{\mathbb P}(1) \vert $, from the morphism $t$ we have
$K \sim -5H + R$. Since $\det \mathcal V \otimes \mathcal O_S(K_S) \cong \mathcal O_S$, the formula for the canonical class of the projective bundle $\mathbb P$ implies 
$K \sim -3H$. Then $R \sim 2H$ and $B = t_*R$ is a quartic.
\end{proof}
$B$ is crucial for our work and a very peculiar quartic threefold, one has: 
 \begin{theorem} The focal locus $B$ is the Igusa quartic threefold and $F = \Sing B$. \end{theorem}
See \cite{CKS}. The Igusa quartic $B$ is a very classical object, see \cite{Ri}.  Its geometry is exposed in chapter V of Baker's treatise \cite{B}.  Its ubiquity in Algebraic Geometry has been mentioned already, see  \cite{CKS, D1, D2, Hu, I, VdG}. This implies that $B$ bears more names, like Castelnuovo - Richmond quartic, and appears in relation to several moduli spaces. Actually, as one can realize reading \cite{C} and \cite{Ri},  $B$ was described at first by Castelnuovo, followed independently by Richmond.  For instance we will use  the Stein factorization of $t$, namely the diagram
\begin{equation} \label{Stein} 
\begin{tikzcd}
{} & {\overline{\mathbb P}} \arrow{dr}{\overline t} \\
{\mathbb P} \arrow{ur}{c} \arrow{rr}{t} && {\mathbb P^4.} 
\end{tikzcd}
\end{equation}
But $\overline {\mathbb P}$ is the moduli space of $6$ ordered points of $\mathbb P^2$, \cite{D1} 9.4.17. On the other hand will see new incarnations of $B$ relating it to Prym curves of genus $6$ and their Prym varieties. \par We recall that $B$ is invariant under the action of the symmetric group $\frak S_6$ so that \begin{equation} \frak S_6 \cong \Aut B \subset \Aut \mathbb P^4, \end{equation} 
cf. \cite{M1} Lemma 5. Let $(z_1: ... :z_6)$ be coordinates in $\mathbb P^5$ and $ s_k := z_1^k + ... + z_6^k$, putting $\mathbb P^4 := \lbrace s_1 = 0 \rbrace$ the equation of $B$ in $\mathbb P^4$ is $ s_4- \frac14 s_2^2 = 0$.  To continue  we  consider a point $y \in S$ and the fibre of $u \vert R: R \to S$ at $y$, say 
\begin{equation} R_y := R \cdot \mathbb P_y. \end{equation}
 Since the ramification divisor $R$ is an element of $\vert \mathcal O_{\mathbb P}(2) \vert$, and no plane is in $B$, then $R_y$ is 
always a conic. Consider the plane $\mathsf P_y = t(\mathbb P_y)$ and $B_y := t_* R_y$. Then it follows  
\begin{equation}
\mathsf P_y \cdot B = 2B_y.
\end{equation}
 Indeed $t^*(\mathsf P_y)$ is reducible, since it contains $\mathbb P_y$, and  the branch divisor of $t:t^*(\mathsf P_y) \to \mathsf P_y$ is $\mathsf P_y \cdot B$. This implies $\mathsf P_y \cdot B = 2B_y$, where $B_y$ is a conic. According to the classical theory $B_y$ is the \it focal conic \rm of $\mathsf P_y$. Finally notice also that $\Aut S \subset \Aut B$. Since $\Aut S \cong \frak S_5$ and $\Aut B \cong \frak S_6$, it  follows that the orbit of $S$ is the union of six Del Pezzo surfaces, say
\begin{equation}  \label {SIXD}
S_{\frak S_6} = \bigcup_{i =  \dots 6}S_i.
\end{equation}
  Each surface $S_i$ is endowed as in (\ref{DUE}) with its universal plane $\mathbb P_i$, $i = 1 ... 6$, and the usual diagram 
\begin{equation} 
\begin{CD}
{\mathbb P^4} @<t_i<< {\mathbb P_i} @>u_i>> S_i. \\
\end{CD}
\end{equation}
 The following well known characterization of $B$ will be also coming into the play.  
\begin{theorem}\label{Segreprimal} $B$ is the strict dual hypersurface of the Segre cubic primal. \end{theorem}
The Segre primal is, modulo $\Aut \mathbb P^{4*}$,  the unique cubic threefold $B^*$ in $\mathbb P^{4*}$ whose singular locus consists of $10$ double points. $10$ is also the maximal number of isolated singular points for a cubic threefold. Moreover the dual map $B^* \dasharrow B$ contracts the $15$ planes of $B^*$ to $15$ singular lines of $B$: $\Sing B$ is their union. $\Sing B^*$ and the $15$ planes of $B^*$ define a $(15_4, 10_6)$-configuration: $4$ nodes in each plane and $6$ planes through each node,  \cite{D3} 2.2. In the Grassmannian $\mathbb G^*$ of lines of $\mathbb P^{4*}$, the Fano variety of lines of $B^*$ is the union of $15$ planes and $6$ Del Pezzo surfaces $S^*_1 ... S^*_6$. These, under the natural isomorphism between $\mathbb G^*$ and $\mathbb G$, are  dual to $S_1 ... S_6$.  Finally we fix the {\sf convention}: $S = S_1$ so that $t = t_1$ and $u = u_1$. \par
Coming to the fundamental locus $F$ of the congruence $S$ its description is classical, cfr. \cite{B}. We have $F = \Sing B$ and the $15$ lines of $F$ intersect $3$ by $3$ along the set of $15$ triple points of $B$. This defines the famous Cremona -Richmond $(15_3, 15_3)$-configuration, \cite{C, Ri}, cfr. \cite{D1} 9.4.4. 
\begin{remark} \rm\label{fundamental}  Let us sketch very briefly a direct reconstruction of $F$. For any $x \in F$ observe that $u_*t^*(x) = \mathbb G_x \cdot S$, where $\mathbb G_x$ is a smooth $4$-dimensional quadric. Then, since $S$ is an integral linear section of $\mathbb G$ and $\dim t^*(x) \geq 1$, the fibre $t^*(x)$ is either a line or a conic in $\lbrace x \rbrace \times S$.   \par Assume $L = u_* t^*(x)$ is a line $L$, then $L$  is one of the \it ten lines of $S$\rm. Hence $L$ is a pencil of planes contained in a given hyperplane of $\mathbb P^4$. Let $A \subset \mathbb P^4$ be its base locus,  then $\mathbb P$ contains the surface $A \times L$ and $t$ contracts it to $A$.  Now assume that $u_*t^*(x)$ is a smooth conic $C$, then $\vert C \vert$ is one of the \it five pencils of conics of $S$\rm. Let $\mathbb P_C = u^*C$,  it turns out that $\mathbb P_C$ is the projectivization of $\mathcal O_{\mathbf P^1}(-1)^{\oplus 2} \oplus \mathcal O_{\mathbb P^1}$ and $t \vert \mathbb P_C $ is its tautological map. Hence $Q_C := t(\mathbb P_C)$ is a rank $4$ quadric and $t \vert \mathbb P_C$ is the small contraction of  $\lbrace x \rbrace \times C \subset \mathbb P_C$ to $x = \Sing Q_C$. Moving $C$ in the pencil of conics $\vert C \vert$, the point $\Sing Q_C$ moves along a line which is in $F$.  In particular each irreducible component of $t^{-1}(F)$ is one of the $15$ surfaces as above, cfr. \cite{CKS}.  \end{remark}
 
\section{Conic bundles associated to a genus $6$ Prym curve}
We now study the linear system of divisors $\vert \mathcal O_{\mathbb P}(2) \vert$.
 \begin{proposition} \label{dimension} For $m \geq1$ one has
 \begin{equation} h^0(\mathcal O_{\mathbb P}(m)) = \binom{m+4}4 + \binom{m+2}4 \ , \ h^i(\mathcal O_{\mathbb P}(m)) = 0 \  for \ i \geq 1. \end{equation} \end{proposition}
We omit the standard proof and only compute $h^0(\mathcal O_{\mathbb P}(m))$ for the future use. Let
 $$ 
\begin{tikzcd}
{} & {\overline{\mathbb P}} \arrow{dr}{\overline t} \\
{\mathbb P} \arrow{ur}{c} \arrow{rr}{t} && {\mathbb P^4.}
\end{tikzcd}
$$
be the Stein factorization of $t$, we consider the birational involution induced by $t$:
\begin{equation} \label{IOTA}
\iota: \mathbb P \to \mathbb P. 
\end{equation}
By Remark (\ref{fundamental}) $c$ is a small contraction, hence the map $c_*: H^0(\mathcal O_{\mathbb P}(m)) \to H^0(\overline t^* \mathcal O_{\mathbb P^4}(m))$
is an isomorphism. Let  $\overline \iota: \overline {\mathbb P} \to \overline{\mathbb P}$ be the biregular involution induced by $\overline t$. In particular
\begin{equation} \iota^* := c^{-1}_* \circ \overline \iota^* \circ c_*: H^0(\mathcal O_{\mathbb P}(m)) \to H^0(\mathcal O_{\mathbb P}(m)) \end{equation}
is an involution with  eigenspaces 
 \begin{equation}
 \label{iota} H^+_m := t^*H^0(\mathcal O_{\mathbb P^4}(m))  \ \text{and}  \ H^-_m :=  t^*H^0(\mathcal O_{\mathbb P^4}(m-2)) \otimes <q_->, 
\end{equation}
where ${div}(q_-) = R = t^{-1}(B)$. This implies the previous equality: $$ h^0(\mathcal O_{\mathbb P}(m)) = \dim H^+_m +  \dim H^-_m =  \binom{m+4}4 + \binom{m+2}4. $$
We denote the elements of $\vert \mathcal O_{\mathbb P}(2) \vert$ by $Q$. Clearly a general fibre of $ u \vert Q$ is a conic. With some abuse we will simply say that $u \vert Q$ is a \it conic bundle structure \rm on $Q$.  From \cite{FV} 2 we have:

\begin{proposition} A general $Q$ is smooth and each fibre of $u \vert Q$ is a conic of rank $\geq 2$. \end{proposition}
 Let $Q$ be general as in the above statement and let $Q_y := (u\vert Q)^*(y)$, then
 \begin{equation}
C := \lbrace y \in S \ \vert \ \rank Q_y = 2 \rbrace
\end{equation}
is a smooth curve endowed  with the embedding sending $y$ to $\Sing Q_y$, we denote it as \begin{equation} \label {Steiner embedding} s: C \to \mathbb P.  \end{equation}   Consider the surface $Q_C = u^*(C)$ then $\Sing Q_C = s(C)$. Moreover it is easy to see that the normalization map $n: \tilde Q_C \to Q_C$ fits in the Cartesian square
\begin{equation}
\begin{CD}
{\tilde Q_C} @>n>> Q_C \\
@V{\tilde u_C}VV @V{u_C}VV \\
{\tilde C}@>{\pi}>> C, \\
\end{CD}
\end{equation}
where $u_C$ is $u \circ n: \tilde Q_C \to C$ and $\tilde u_C$ is the Stein factorization of $u_C$. Then $\pi$ is an \'etale $2:1$ cover, defined by a $2$-torsion element $\eta \in \Pic C$. The next theorem summarizes our situation.
\begin{theorem} \label{CB} Let $Q$ be general as above:
\begin{enumerate}
\item $C \in \vert -2K_S \vert$,
\item $\eta$ is non trivial,
\item $C$ is a general genus $6$ curve.
\item $u: C \to S \subset \mathbb P^5$ is the canonical embedding,
\item $t \circ s: C \to \mathbb P^4$ is defined by $\omega_C \otimes \eta$.
 \end{enumerate}
\end{theorem}
\begin{proof} Some basic results on conic bundles and genus $6$ curves imply the statement. (1) is given by the formula for the discriminant of a conic bundle, see \cite{Sa}. (2), (3), (4) are known properties one retrieves in \cite{FV} 2. (5) We summarize a well known proof. Consider the section $\tilde C := n^*s(C)$ of $\tilde u_C: \tilde Q_C \to \tilde C$ and let $n_{\tilde C} := n \vert \tilde C$. Then, by adjunction formula on $\tilde Q_C$, the line
bundle $n_{\tilde C}^* \mathcal O_{\mathbb P}(1)$ is $\omega_{\tilde C}$. Moreover one can show that $n_{\tilde C}^*H^0(\mathcal O_{\mathbb P}(1))$ is an eigenspace of the involution defined by $\pi$. Since its dimension is $5$ then $n_{\tilde C}^*H^0(\mathcal O_{\mathbb P}(1)) = \pi^*H^0(\omega_C \otimes \eta)$. This implies $\omega_C \otimes \eta \cong \mathcal O_{\mathbb P}(1)$.  \end{proof}
We assume that $\omega_C \otimes \eta$ is very ample, since it is true for a general $C$ of genus $6$, \cite{DS}. We will also say that $(C,\eta)$ is the \it discriminant \rm of $u: Q \to S$. \rm  Now $(C,\eta)$ defines a point in the moduli space $\mathcal R_6$ of genus $6$ Prym curves. Starting from such a point we are interested in the set of possible realizations of $(C, \eta)$ as the discriminant of a 
conic bundle $Q \in \vert \mathcal O_{\mathbb P}(2) \vert$.  \par We say that $Q, Q' \in \vert \mathcal O_{\mathbb P}(2) \vert$ are isomorphic if there exists a biregular map $f: Q \to Q'$  such that $u \vert Q = u' \vert Q' \circ f$. We are interested to study the isomorphism classes of $\vert \mathcal O_{\mathbb P}(2) \vert$.  Let  $\mathcal I_{\mathsf e}$ be the ideal sheaf of a set $\mathsf e$ of four general points in $\mathbb P^2 \times \mathbb P^2$, to this purpose we will use the main diagram of rational maps in \cite{FV} p. 529 and the induced linear isomorphism $$ \vert \mathcal O_{\mathbb P}(2) \cong \vert \mathcal I_{\mathsf e}^2(2,2) \vert. $$  Let us give a more precise explanation. This isomorphism depends on the choice of a contraction \begin{equation} \label{SIGMA} \sigma: S \to \mathbb P^2 \end{equation} of four disjoint exceptional lines. Let $e_i = \sigma (E_i)$, $i = 1 \dots 4$, $E_i$ being an exceptional line. Then, setting $\mathsf e := \lbrace e_1 ... e_4 \rbrace$, we have a set of points in general linear position and the same is true for its diagonal embedding $\mathsf e \subset \mathbb P^2 \times \mathbb P^2$. Notice also that $\dim \vert \mathcal I_{\mathsf e}(1,1) \vert = 4$.
\begin{definition} $t_{\mathsf e}: \mathbb P^2 \times \mathbb P^2 \dasharrow \mathbb P^4$ is the rational map associated to $\vert \mathcal I_{\mathsf e}(1,1) \vert$.
 \end{definition} It is easy to see that $\deg t_{\mathsf e}= 2$. The same is true replacing $t_{\mathsf e}$ by   
\begin{equation}
t_E := t_{\mathsf e} \circ (\sigma \times id_{\mathbb P^2}):S \times \mathbb P^2 \dasharrow \mathbb P^4.
\end{equation}
The strict relation between the tautological morphism $t$ and $t_{\mathsf e}$ is described in the following diagram that is mentioned in \cite{FV}, 2 p.529. This can be written as follows:
\begin{equation} \label{Main CD}
\begin{tikzcd}[column sep=6pc]
  &{\widetilde{\mathbb P}}  \arrow{d}{\epsilon_{_2}}  \arrow{r} {\epsilon_{_1}}& {\mathbb P}   \arrow{d}{t}   \arrow{r}{u} &{S} \\
  {\mathbb P^2 \times \mathbb P^2}\arrow{r}{\sigma^{-1} \times id_{\mathbb P^2}}     &
  {S \times \mathbb P^2}  \arrow{ru}{\epsilon} \arrow{r}{t_E} & {\mathbb P^4} &\\
  \end{tikzcd}
\end{equation}
Here  $\epsilon_{_1}, \epsilon_{_2}$ are birational morphisms resolving the indeterminacy of the birational map
\begin{equation}
\epsilon := \epsilon_{_1} \circ \epsilon_{_2}^{-1}: S \times \mathbb P^2 \dasharrow \mathbb P.
\end{equation}
Let $L \in \vert \sigma^* \mathcal O_{\mathbb P^2}(1) \vert$, as in \cite{FV} 2 (3) the birational map $\epsilon$ can be defined via the exact sequence
$$
0 \to \mathcal V \stackrel {j} \to H^0(\mathcal O_S(L)) \otimes \mathcal O_S(L) \stackrel{r} \to \bigoplus_{i  = 1 \dots 4} \mathcal O_{E_i}(2L) \to 0
$$
where $E_i = \sigma^{-1}(e_i)$ is the exceptional line of $\sigma$ over $e_i$, $r$ is the natural multiplication of sections and $\mathcal V = \Ker r$ is the previously considered
rank $3$ Mukai bundle of $S$. Dualizing and projectivizing  the morphism of locally free sheaves $j$ one obtains $\epsilon: S \times \mathbb P^2 \dasharrow \mathbb P$. 
\begin{remark} \rm It is useful to add some remarks on $\epsilon_1$ and $\epsilon_2$. It turns out that  $\mathbb P_{E_i} := u^{-1}(E_i)$  is biregular to the blowing up of $\mathbb P^3$ along a line.
Moreover let $F_i \subset \mathbb P_{E_i}$ be the exceptional divisor of such a blowing up, then $F_i = \mathbb P^1 \times E_i$ and $u \vert F_i$ is the projection onto $E_i$. We have: \par
\bf (a) \rm  $\epsilon_{_2}: \widetilde {\mathbb P} \to S \times \mathbb P^2$ is the blowing up of the disjoint union $F_1 \cup ... \cup F_4$.  \ \bf (b) \rm
 Consider the threefold $\tilde F_i := \epsilon_{_2}^{-1}(F_i)$ then each fibre of $u \circ \epsilon_{_1}: \tilde F_i \to E_i$ is $\mathbb P^2$ blown up in a point. 
\ \bf (c) \rm  The contraction $\epsilon_{_1}(\tilde F_i)$ of $\tilde F_i$  is $E_i \times \mathbb P^2$ and $\epsilon_{_1}: \tilde F_i \to E_i \times \mathbb P^2$ is the blowing up of $E_i\times \lbrace e_i \rbrace$.
\ \bf (d) \rm Let $\mathbb {P}'_{E_i}$ be the strict transform of $\mathbb P_{E_i}$ by $\epsilon_{_2}$, then $\epsilon_{_1}$ contracts ${\mathbb P}'_{E_i}$ to  $E_i \times \lbrace e_i \rbrace$.
 \end{remark}
 Now let $\phi:=  (\sigma \times id_{\mathbb P^2}) \circ \epsilon$, we consider the birational map
 \begin{equation} \label{QUATTRO}
  \phi : \mathbb P \dasharrow \mathbb P^2 \times \mathbb P^2.
\end{equation}
Let $Q \in \vert \mathcal O_{\mathbb P}(2) \vert$ be general and let $Q' \subset \mathbb P^2 \times \mathbb P^2$ be  its strict transform by $ \phi$. To understand $Q'$ consider $Q_{E_i} := Q \cdot \mathbb P_{E_i}$. It  turns out that $Q_{E_i}$ is $\mathbb P^1 \times \mathbb P^1$ blown up in two distinct points $o_1,o_2$ and that $u: Q_{E_i} \to E_i$ is the conic bundle defined by the pencil $\vert \mathcal  I_{o_1,o_2} (1,1) \vert$, where $\mathcal I_{o_1, o_2}$ is the ideal sheaf of $\lbrace o_1, o_2 \rbrace$ in $Q_{E_i}$.  Now let $E'_1, E'_2$ be the exceptional lines of $Q_{E_i}$ over $o_1, o_2$, then $\phi$ blows $E'_1, E'_2$ up   and flops down the resulting surfaces to two lines. These surfaces intersect at  a double point of $Q'$ which is the contraction of $Q_{E_i}$. Moreover $\phi: Q \to Q'$ is biregular on  $U = Q \setminus (\bigcup_{i = 1 ... 4} Q_{E_i})$. 
 The resulting $Q'$ is singular at $\mathsf e$ and has bidegree $(2,2)$. Let $\mathcal I_{\mathsf e}$ be the ideal sheaf of $\mathsf e$ in $\mathbb P^2 \times \mathbb P^2$, then
 $Q' \in \vert \mathcal I_{\mathsf e}(2,2) \vert$.  The map $\phi$ is studied in \cite{FV} 2. We conclude by pointing out the obvious consequence of our discussion, which will be a useful tool. Let
    \begin{equation} \label{h}
h:= (\sigma \times id_{\mathbb P^2})_* \circ \epsilon_{1*} \circ \epsilon_{_1}^* ,
\end{equation}
 
\begin{theorem} \label {Translation theorem} $h: H^0(\mathcal O_{\mathbb P}(2)) \to H^0(\mathcal I^2_{\mathsf e}(2,2))$ is an isomorphism.  \end{theorem}
   \section{Igusa pencils and $E_6$-quartic threefolds}
{   Now we introduce the useful notion of an \it Igusa pencil of quartic threefolds. \rm As in (\ref{IOTA}) consider the birational involution $\iota: \mathbb P \to \mathbb P$ induced by the morphism  $t: \mathbb P \to \mathbb P^4$. As in above let $R = t^{-1}(B)$, we know that $\iota^*$ acts linearly on $\vert \mathcal O_{\mathbb P}(2) \vert$ so that its set of fixed points is
\begin{equation}
t^* \vert \mathcal O_{\mathbb P^4 }(2) \vert  \ \cup \ \lbrace R \rbrace.
\end{equation}
Let $\tilde P \subset \vert \mathcal O_{\mathbb P}(2) \vert$ be a $\iota^*$-invariant pencil on which $\iota^*$ is not the identity, then $\iota^*$ acts on $\tilde P$ with exactly two fixed elements. These  are $R$ and $t^*A$, where $A \subset \mathbb P^4$ is a quadric. Equivalently $\tilde P$  is the pull-back by $t$ of the pencil of quartic threefolds $P$, generated by $B$ and the \it double quadric \rm $2A$. Let $A = {\rm  div}(a)$ and ${\rm div}(b) = B$, the equation of $P$ is
\begin{equation}
P = \lambda a^2 + \mu b = 0.
\end{equation}}
\begin{definition} We say that $P$ is an Igusa pencil of quartic threefolds. \end{definition}
\begin{definition} We say that a threefold $X \in P$ is an $E_6$-quartic threefold. \end{definition}
  The family of these pencils is parametrized by $\vert \mathcal O_{\mathbb P^4}(2) \vert$. Assuming $P$ \it general \rm and $X$ \it general \rm  in $P$, let us see some properties of $X$ and of $P$. For $Q \in \vert \mathcal O_{\mathbb P}(2) \vert$ we will set $\overline Q := \iota^*Q$. At first we emphasize that for a general $X$ the threefold $t^*X$ is split:
$$
 t^*X = Q + \overline Q,
$$  where $Q, \overline Q$ are \it smooth \rm and $Q \neq \overline Q$. The base scheme of $P$ is the  \it contact surface \rm
$$
\lbrace a^2 = b = 0 \rbrace = 2A \cdot X \subset \mathbb P^4,
$$
where $A = {\rm div }(a)$ is a smooth quadric transversal to $B$. Let $y \in S$ and $\mathsf P_y \subset \mathbb P^4$ the corresponding plane. We know that   $\mathsf P_y \cdot B = 2B_y$, where $B_y$ is a conic. Let $b_y$ be the equation of $B_y$ in $\mathsf P_y$. Restricting $\lambda a^2 + \mu b$ to $\mathsf P_y$ we obtain the sum of  squares 
$$
\lambda a^2 + \mu b^2_y.
$$
This defines the union of two conics of $\mathsf P_y$, namely $t_* Q_y$ and $t_* \overline Q_y$. Furthermore the covering $
t: t^{-1}(\mathsf P_y) \to \mathsf P_y
$
splits, since it is branched on $2B_y$.  Let $\overline {\mathbb P}_y := \iota^*(\mathbb P_y)$, then $$t^*(\mathsf P_y) = \mathbb P_y + \overline{\mathbb P}_y. $$ In particular $[\mathbb P_y + \overline {\mathbb P}_y] = [H^2]$  in $CH^2(\mathbb P)$ and $\overline {\mathbb P}_y$ is  a birational section of $u: \mathbb P \to S$. 
\begin{theorem} Assume $X$ and $P$ are general as above then:
\begin{enumerate}
\item $\Sing X$ is $A \cap \Sing B$ and consists of $30$ ordinary double points.\smallskip  
\item The discriminants of $u \vert Q$, $u \vert \overline Q$ are smooth genus $6$ Prym curves.\smallskip
\item $t \vert Q: Q \to X$ and $t \vert \overline Q: \to X$ are small contractions to $X$.
\end{enumerate}
\end{theorem}
 \begin{proof} By Bertini theorem $X \in P$ is smooth outside the surface $B_a := \lbrace a = b = 0 \rbrace$. Moreover $\Sing B_a = \Sing B \cap A$ and this set
 consists of $30$ ordinary double points of $X$. Let $x \in B_a - \Sing B_a$ then $B$ is smooth at $x$ and hence $\lbrace a^2 = 0 \rbrace$ is the
 unique element of $P = \lbrace \lambda a^2 + \mu b = 0 \rbrace$ which is singular at $x$. This implies (1). (2) follows from theorem (\ref{CB}) since $Q$, $\overline Q$ are general
 in $\vert \mathcal O_{\mathbb P}(2) \vert$. Then they are smooth with no fibre of rank $\leq 1$. For (3) it suffices to recall that the fibre of $t$ at a general $x \in \Sing B$ is $\mathbb P^1$.
  \end{proof}
Let $Q$, $\overline Q$ be as above, consider the Cartesian square of \it birational morphisms \rm
\begin{equation} \label{JX}
\begin{CD}
{\tilde X}@>>> Q \\
@VVV @V{t \vert Q}VV \\
{\overline Q} @>{t \vert  \overline Q}>>X. \\
\end{CD}
\end{equation}
Since $X$ is general the set $\Sing X = X \cap \Sing B$   is general in $\Sing B$ and, by theorem (\ref{Fundamental}),    the fibre of $t \vert Q$ and $t \vert \overline Q$ at $o \in \Sing X$ is $\mathbb P^1$. Let $\beta: \tilde X \to X$ be the birational morphism
induced by the diagram, then $\beta^*(o) = \mathbb P^1 \times \mathbb P^1$  and   $\beta$ is the blowing up of $\Sing X$. $\tilde X$ is smooth and, denoting by $T$ the intermediate Jacobian of $T$, one has:
 \begin{equation}
JQ \cong J \tilde X \cong J\overline Q.
\end{equation}
See \cite{CG}, lemma 3.11. Finally let $(C, \eta)$ and $(\overline C, \overline \eta)$ be the discriminants of $u \vert Q$ and $u \vert \overline Q$. Consider the corresponding Prym varieties $P(C,\eta)$ and $P(\overline C, \overline \eta)$ as in \cite{B1}, then
\begin{equation}
JQ \cong P(C, \eta) \ , \ J\overline Q \cong P(\overline C, \overline \eta).
\end{equation}
Now let us recall that $t: \mathbb P \to \mathbb P^4$ is just one of the six double coverings $$ t_i: \mathbb P_i \to \mathbb P^4, \ i = 1 ... 6, $$ considered in
(\ref{SIXD}). Hence $t^*X = Q + \overline Q$ is just one of the splittings
\begin{equation}
t_i^*X = Q_i + \overline Q_i \subset \mathbb P_i, \ i = 1 ... 6.
\end{equation}
Hence, birationally,  $X$ is naturally endowed with six pairs of conic bundle structures:
\begin{equation}
\lbrace u_i\vert Q_i: Q_i \to S_i   \  , \   u_i\vert \overline Q_i: \overline Q_i \to S_i \rbrace.
\end{equation}
 To simplify our notation we set:
\begin{equation}
 \ \ u_i := u_i \vert Q_i \ , \ \overline u_i := u_i \vert \overline Q_i.
\end{equation}
 \begin{definition} \label{SIX}  $\lbrace u_i \ , \  \overline u_i \ \ \ i = 1...6 \rbrace$ is the double six of conic bundle structures of $X$. \end{definition}
The set of  corresponding discriminants is  \it the double six of Prym curves of $X$\rm:  
 \begin{equation} \label{DSX} \lbrace  (C_i, \eta_i) \ , \ (\overline C_i, \overline \eta_i), \ \ \ i = 1 ... 6 \rbrace.
\end{equation} \par
\begin{remark} \rm One recognizes here the shadow of the set of  $27$ lines on the cubic surface, via its $36$ Schl\"afli double sixs.   Let
$ \frak p: \mathcal R_6 \to \mathcal A_5$ be the Prym map,
sending $(C,\eta)$ to its Prym $P(C, \eta)$, \cite{DS, D1}. $\frak p$ has degree
$27$. The Weyl group of $E_6$, preserving the incidence 
of the $27$ lines, is its mondromy group. So far we see that \it the double six of $X$ defines a subset of the fibre of $\frak p$ at the moduli point of $J\tilde X$ in $\mathcal A_5$.  \end{remark} 
   \section{Moduli of Prym sextics of genus $6$}
It is time to address the rationality results for {  some} moduli spaces related to $\mathcal R_6$ and their relations. This involves the moduli of $4$-nodal Prym
plane sextics of genus six. 
With this in mind we start from a general genus $6$ Prym curve $(C, \eta)$. We know from theorem \ref{CB} and \cite{FV} that $(C, \eta)$ is biregular to the discriminant of a smooth $Q \in \vert \mathcal O_{\mathbb P}(2) \vert$. According to the same theorem we know more: consider $ u \vert Q: Q \to S \subset \mathbb P^5 $
then $C$ is canonical in $\mathbb P^5$.  Moreover let \begin{equation} s: C \to \mathbb P \end{equation} be the map sending $y \in C$ to $\Sing Q_y$, then $s \circ t: C \to \mathbb P^4$ is the Prym canonical map of $(C, \eta)$.  Consider $C_s := s(C)$ and its ideal sheaf $\mathcal I_{C_s}(2)$ in $\mathbb P$, we have the standard exact sequence
$$ 
0 \to \mathcal I_{C_s}(2) \to \mathcal O_{\mathbb P}(2) \to \mathcal O_{C_s}(2) \to 0.
$$
Notice also that $\mathcal O_{C_s}(1) \cong \omega_C \otimes \eta$ and $\mathcal O_{C_s}(2) \cong \omega_C^{\otimes 2}$. Let $(C,\eta)$ be general then:
\begin{proposition} \label {4.1} One has  $h^0(\mathcal I_{C_s}(2)) = 1$ and $h^i(\mathcal I_{C_s}(2)) = 0$, $i \geq 1$. \end{proposition}
\begin{proof}  We have $h^0(\mathcal O_{C_s}(2)) = 15$ and $h^0(\mathcal O_{\mathbb P}(2)) = 16$ from \ref{dimension}. Hence the previous exact sequence implies 
$h^0(\mathcal I_{C_s}(2)) \geq 1$. Assume the latter inequality is strict, then it follows that $\dim   t^* H^0(\mathcal  O_{\mathbb P}(2))   \cap H^0(\mathcal I_{C_s}(2)) \geq 1$ in $H^0(\mathcal O_{\mathbb P}(2))$. Equivalently a quadric contains the Prym canonical model $t(C_s)$. But then the moduli point of $(C,\eta)$ in $\mathcal R_6$  is in the ramification of the Prym map $\frak p$, \cite {B1} 7.5, and $(C, \eta)$ is not general: a contradiction.
\end{proof}
Let $\mathbb P_C \subset \mathbb P$ be the universal plane over $C$. Then $\mathbb P_C$ is defined by the embedding of $C$ in $\mathbb G$ via  $\mathcal V_C := \mathcal V \otimes \mathcal O_C$, where $\mathcal V$ is $\mathsf u_* \mathsf t^* \mathcal O_{\mathbb P^4}(1)$. Then $\mathcal V_C$  is uniquely defined as the kernel of the evaluation of global sections of the \it unique \rm rank two \it Mukai bundle \rm over $C$, see \cite{M} section 5 for its definition and uniqueness. Moreover the map $s$ is uniquely defined by an exact sequence
$$
0 \to \mathcal N_C \to \mathcal V_C \stackrel {v_s} \to \omega_C \otimes \eta \to 0
$$
for a given $1$-dimensional space $< v_s >  \ \subset  {\rm Hom}(\mathcal V_C, \omega_C \otimes \eta)$.  
Conversely, starting from a general $ v \in {\rm Hom}(\mathcal V_C, \omega_C \otimes \eta)$, one obtains an exact sequence as above and a section $s_v: C \to \mathbb P_C \subset \mathbb P$. For a general $v$, $s_v(C)$ is certainly contained in a unique $Q_v \in \vert \mathcal O_{\mathbb P}(2) \vert$, as prescribed in (\ref{4.1}). However $C$ does not necessarily coincide with the discriminant curve of the conic bundle $u \vert Q_v$ nor $s_v: C \to Q_v$ is the map sending $y \in C$ to the singular point of the fibre of $u \vert Q_v$ at $y$. This prompts us for the next definition. Let $(C, \eta, s)$ be \it any \rm triple such that:
\begin{enumerate} \it
\item $C \in \vert \mathcal O_S(2) \vert$ is smooth and $(C, \eta)$ is a Prym curve,
\item {  $s: C \to \mathbb P_C$ is a section of the $\mathbb P^2$-bundle $u\vert \mathbb P: \mathbb P_C \to C$,
\item $s$ is defined by an epimorphism  $v \in {\rm Hom}(\mathcal V_C, \omega_C \otimes \eta)$,}
\item a unique $Q \in \vert \mathcal O_{\mathbb P}(2) \vert$ contains $s(C)$.
\end{enumerate}
\begin{definition} \label{Steiner} $(C, \eta, s)$ is a Steiner map of the pair $(C, \eta)$ if the discriminant of $u \vert Q$ is the pair $(C, \eta)$ and $s: C \to Q$ is the map sending $y \in C$ to the singular point of 
${u \vert Q}^*(y)$.
\end{definition}
The name is motivated by some relation to the classical notion of \it Steinerian\rm, the locus of singular points of the quadratic polars of a  hypersurface,  \cite{D}, 1.1.6. If $(C,\eta,s)$ is a Steiner map, $Q$ is smooth and each fibre of $u \vert Q$ has rank $\geq 2$, cfr. \cite{FV} 2. Now a general $Q \in \vert \mathcal O_{\mathbb P}(2) \vert$ 
defines a triple $(C, \eta, s)$ satisfying the above conditions and such that $s$ is the Steiner map of $u \vert Q$. Therefore, giving a Steiner map $(C, \eta, s)$, is equivalent
to give a general $Q$ with discriminant $(C, \eta)$ and Steiner map equal to $s$.
  \par Let $t := (C,\eta,s)$, $t' := (C',\eta',s')$ be two Steiner maps and $u \vert Q$, $u \vert Q'$ their respectively associated conic bundles. We will say that \it $t$ and $t'$ are isomorphic \rm if there exists a biregular map $\alpha: Q \to Q'$ such that $u \circ \alpha = u \vert Q$. It is easy to see that $\alpha^* \mathcal O_{Q'}(1) \cong \mathcal O_Q(1)$, so that $\alpha = a \vert Q$ for some $a \in \Aut \mathbb P$. To construct the moduli of Steiner maps, somehow abusing of the word \it moduli\rm, we rely on Geometric Invariant Theory and construct the GIT quotient $\vert \mathcal O_{\mathbb P}(2) \vert \dslash \Aut \mathbb P$. Since $\mathbb P$ is the universal plane of $\mathbb P^4$ over $S$, then $\Aut \mathbb P$ is the subgroup of 
$\Aut \mathbb P^4$ leaving invariant the congruence of planes $\lbrace \mathsf P_x, x \in S \rbrace$. This is actually the symmetric group $\frak S_5 \cong \Aut S$. 
  
 \begin{definition} {$\mathcal R^{cb}_6:=  \vert \mathcal O_{\mathbb P}(2) \vert \dslash \Aut \mathbb P$} is the moduli space of Steiner maps. \end{definition} 
By \cite{MF} 0.2 (2) the GIT quotient $\mathcal R^{cb}_6$ is reduced, irreducible and normal. $\mathcal R^{cb}_6$ has dimension $15$ and dominates $\mathcal R_6$ via the forgetful map, \cite{FV}.  In \cite{SB}  Shepherd Barron proves that the quotient $\mathbb P(V) \dslash \frak S_5$ of any finite representation $V$ of $\frak S_5$ is rational. This implies the next theorem.
   \begin{theorem} \label{NEWRATIONAL} $\mathcal R^{cb}_6$ is rational. \end{theorem}
We recall from (\ref{Steiner}) that giving a general marked $E_6$-quartic $(X,Q)$ is equivalent to give its Steiner map $(C, \eta, s)$. Hence $\mathcal R^{cb}_6$ is also the GIT quotient parametrizing
 $\Aut \mathbb P$ classes of marked $E_6$-quartics.  Continuing in the same vein let $\sigma: S \to \mathbb P^2$ be a fixed contraction
 of four lines and $\Aut_{\sigma} \subset \Aut S$ the stabilizer of its exceptional divisor: 
 \begin{definition} $  \widetilde{\mathcal R}^{cb}_6   = \vert \mathcal O_{\mathbb P}(2) \dslash \Aut_{\sigma} S$. \end{definition}      \begin{theorem} \label{NEWRATIONAL} $  \widetilde{\mathcal R}^{cb}_6  $ is rational. \end{theorem}
\begin{proof} We give a standard proof related to the geometry of $\vert \mathcal O_{\mathbb P}(2) \vert$ and similar to theorem 6 in \cite{SB}. $\frak S_4$  is isomorphic to $\Aut_{\sigma} S$ and acts on the exceptional divisor $E := E_1 + E_2 + E_3 + E_4$ of $\sigma$ as the permutation group of its irreducible components. $E_i$ is a line in the Grassmannian $\mathbb G$ that is a pencil of planes in $\mathbb P^4$. Let $u: P_i \to E_i$ be the universal plane and $T_i := t(P_i)$. Then $T_i$ is a hyperplane and $t \vert T_i: P_i \to T_i$ is the contraction of $E_i \times F_i$ to $F_i$, where $F_i$ is the base line of the pencil $E_i$.  
Clearly $\frak S_4$ acts on the set $\lbrace T_1 \dots T_4 \rbrace$. More precisely $\frak S_4 = \Aut_{\sigma}S$ is the stabilizer of the point $o_{\sigma} = T_1 \cap \dots \cap T_4$ under the action of $\Aut S$ on $\mathbb P^4$.  $\frak S_4$ acts in the same way on the set $\lbrace 2P_1 \dots 2P_4 \rbrace \subset \vert \mathcal O_{\mathbb P}(2) \vert$ and this spans $\mathbb P(W)$,
where $W \subset V := H^0(\mathcal O_{\mathbb P}(2) \vert$ is an isomorphic representation of dimension $4$. Passing to the dual spaces and projectivizing we obtain a linear projection $\pi: \mathbb P(V^*) \to \mathbb P(W^*)$.  After the blowing up $\gamma$ of its center, we obtain a $\mathbb P^{12}$-bundle $p: \mathbb P_{\sigma} \to \mathbb P(W^*)$ and the $\frak S_4$-linearized line bundle $\gamma^* \mathcal O_{\mathbb P(V^*)}(1)$. By the same proof of \cite{SB} theorem 6, this descends to a $\mathbb P^{12}$-bundle on $\mathbb P^3 \dslash \frak S_4$, which is rational by the symmetric functions theorem. This implies the statement. \end{proof} 
 \rm  Now the inclusion $\Aut S_{\sigma} \subset \Aut S$ induces a morphism $\tilde d: \widetilde {\mathcal R}^{cb}_6 \to \mathcal R^{cb}_6$ of degree $5$.
Let $(C, \eta, s)$ be a Steiner map and let  $M := \sigma^* \mathcal O_{\mathbb P^2}(1) \otimes \mathcal O_C$. Then $M$ defines the plane sextic model $\sigma(C)$ and it is one of the five elements of the Brill-Noether locus \begin{equation} W^2_6(C) := \lbrace M \in \Pic^6 C \ \vert \ h^0(M) \geq 3 \rbrace. \end{equation} 
Consider the set of $4$-tuples  $\lbrace (C,\eta, s, M), \ M \in W^2_6(C) \rbrace$ and its orbit $F$ under the action of $\Aut S$. It is easy to see that the action of $\Aut_{\sigma} S$ on $F$ leaves each $M$ invariant. In particular this establishes a natural bijection between $W^2_6(C)$ and the fibre of $\tilde d$ at the isomorphism class of $(C, \eta,s)$. Then the next theorem follows.
  \begin{theorem} $\widetilde {\mathcal R}_6^{cb}$ is the parameter space for the $\Aut_{\sigma} S$-isomorphism classes of the fourtuples $(C, \eta, s, M)$ and $\tilde d: \widetilde{\mathcal R}^{cb}_6 \to \mathcal R^{cb}_6$ is the forgetful map.\end{theorem}

\begin{definition} A Prym plane sextic of genus $6$ is a pair $(C', \eta')$ such that:
\begin{enumerate}
\item  $C' \subset \mathbb P^2$ is a $4$-nodal plane sextic, 
\item $\Sing C'$ is a set of $4$ points in general position,
\item  $\eta' \in \Pic C'$ is a non zero $2$-torsion element
\item  $\eta'$ is endowed with an isomorphism $\eta'^{\otimes 2} \cong \mathcal O_{C'}$.
\end{enumerate}
\end{definition}
 \begin{definition} $\mathcal R^{ps}_6$ is the moduli space of genus $6$ Prym sextics $(C', \eta')$. \end{definition}
 $C'$ is \it integral \rm and \it stable \rm of arithmetic genus $10$. Let $\overline{\mathcal R}_g$ be the compactification of $\mathcal R_g$ as in \cite{FL}, then $(C', \eta')$ defines a point in the boundary $\overline {\mathcal R}_{10} \setminus \mathcal R_{10}$. This assignment defines a generically injective rational map $\mathcal R^{ps}_6 \dasharrow \overline {\mathcal R}_{10}$.  $\mathcal R^{ps}_6$ is 
  irreducible, \cite{FV} 1.  In the sequel we prove that $\mathcal R^{ps}_6$ and $  \widetilde{\mathcal R}^{cb}_6  $ are birational.  We will see that a $4$-tuple
 $(C, \eta, s, M)$ defines uniquely  a pair $(C', \eta')$ up to isomorphisms and conversely.   Let 
$ \nu: C \to C' $ be the normalization map, $\nu^*$ defines the exact sequence of $2$-torsion groups
\begin{equation} \label{SEXTIC}
0 \to \mathbb Z_2^4 \to \Pic_2 C' \stackrel{\nu^*} \to \Pic_2 C \to 0.
\end{equation}
It is indeed well known that $\Ker \nu^*$ is determined by the $4$-nodes of $C'$ as ${\mathbf C^*}^4$.  
\begin{definition} We fix the notation
$ \nu^* \eta' =: \eta$.
\end{definition}
As in above, let us denote by $\mathsf e \subset \mathbb P^2 \times \mathbb P^2$ the diagonal embedding of $\Sing C'$ and by $\mathcal I_{\mathsf e}$ its ideal sheaf. Our goal is now to prove the following theorem.
\begin{theorem} \label{BIRATIONAL} $  \widetilde{\mathcal R}^{cb}_6  $ and $\mathcal R^{ps}_6$ are birational. \end{theorem}
The proof relies on \cite{FV} and \cite{B2}. We use diagram (\ref{Main CD}) and the isomorphism 
$$
h: H^0(\mathcal O_{\mathbb P}(2)) \to H^0(\mathcal I^2_{\mathsf e}(2,2)),
$$
defined in (\ref{h}). We assume that $Q' \in \vert \mathcal I^2_{\mathsf e}(2,2) \vert$ is \it general \rm so that $\Sing Q'$ consists of $4$ nodes at $\mathsf e$. Let $(z_1:z_2:z_3)$ be coordinates in $\mathbb P^2$ then the equation of $Q'$ is
\begin{equation} \label{equation cb}
\begin{pmatrix}{z_1 \ \ z_2 \ \ z_3}\end{pmatrix} {\huge A} \begin{pmatrix} z_1 \\ z_2 \\ z_3 \end{pmatrix} = 0,
\end{equation} 
where $\huge A$ is a symmetric matrix of quadratic forms. Let $u': \mathbb P^2 \times \mathbb P^2 \to \mathbb P^2$ be the first projection then $u' \vert Q': Q' \to \mathbb P^2$ is a conic bundle. 
Since $Q'$ is general every fibre of $u' \vert Q'$ is a conic of $\rank \geq 2$ and the curve $C' := \lbrace \det A = 0 \rbrace$  is a $4$-nodal sextic, singular at $u'(\mathsf e)$. For such a general $Q'$, the matrix $\huge A$ defines an exact sequence of vector bundles
\begin{equation} \label{exact sequence}
0 \to \mathcal O_{\mathbb P^2}(-4)^{\oplus 3} \stackrel{A} \to \mathcal O_{\mathbb P^2}(-2)^{\oplus 3} \to \eta' \to 0
\end{equation}
where $\eta'$ is a non  zero $2$-torsion element of $\Pic C'$. $(C',\eta')$ is a genus $6$ Prym plane sextic. Conversely a general Prym plane sextic $(C', \eta')$ uniquely defines a resolution of $\eta'$ as above up to isomorphisms. See \cite{B2} thm. B and \cite{FV} 1. The resolution of $\eta'$ defines a conic bundle $Q' \in \vert \mathcal I^2_{\mathsf e}(2,2) \vert$ as in (\ref{equation cb}), up to the action of $\PGL(3)$ on $\huge A$.  \begin{remark} \rm Let $s': C' \to \mathbb P^2 \times \mathbb P^2$ be the Steiner map of $u' \vert Q'$, $s'$ is an embedding, \cite{FV} 2. Let $M := s'^{*}\mathcal O_{\mathbb P^2 \times \mathbb P^2}(1,0)$, the sequence  (\ref{exact sequence}) implies $\label{line bundles} \mathcal O_{\mathbb P^2 \times \mathbb P^2}(0,2) \cong \eta' \otimes M^{\otimes 2}$.
   \end{remark}
 \begin{proof}[\sf Proof of theorem (\ref{BIRATIONAL})]  From diagram (\ref{Main CD}) we have the commutative diagram 
\begin{equation}
\begin{tikzcd}
{\mathbb P^2 \times \mathbb P^2} \arrow[r, dashed, "\phi"] \arrow[d, "u'"]
& {\mathbb P}  \arrow[d, "u"] \\
{\mathbb P^2} \arrow[r, dashed, "\sigma^{-1}"].
& S
\end{tikzcd}
\end{equation}
In it $u'$ is the projection onto the first factor. Moreover we have  $\phi = \epsilon \circ  (\sigma^{-1} \times id_{\mathbb P^2})$ as in (\ref{QUATTRO}) and  $\epsilon = \epsilon_2 \circ \epsilon_1^{-1}$ as in (\ref{Main CD}). In particular the strict transform by $\phi$
of a \it general \rm element $Q \in \vert \mathcal O_{\mathbb P}(2) \vert$ is obtained applying to $Q$ the projective isomorphism
\begin{equation}
\phi': \vert \mathcal O_{\mathbb P}(2) \vert \to \vert \mathcal I^2_{\mathsf e}(2,2) \vert,
\end{equation}
associated to the isomorphism $h^{-1}: H^0(\mathcal O_{\mathbb P}(2)) \to H^0(\mathcal I^2_{\mathsf e}(2,2))$ of theorem \ref{Translation theorem}.
Let $Q' = \phi'(Q)$ be the strict transform of $Q$, then $\phi \vert Q': Q' \to Q$ is birational. Moreover the conic bundle $u' \vert Q'$ is birationally equivalent to $u \vert Q$ and fits in the diagram
diagram  
$$
\begin{tikzcd}
{Q'} \arrow[r, dashed, "\phi \vert Q'"] \arrow[d, "u'"]
& {Q}  \arrow[d, "u"] \\
{\mathbb P^2} \arrow[r, dashed, "\sigma^{-1}"]
& S
\end{tikzcd}
$$
Now, following a typical method under these circumstances, we use the above discussion to define  rational maps $\alpha: \widetilde{\mathcal R}^{cb}_6 \dasharrow \mathcal R^{ps}_6$ and $\beta: \mathcal R^{ps}_6 \dasharrow \widetilde{\mathcal R}^{cb}_6$ which are inverse to each other. Since $\widetilde {\mathcal R}^{cb}_6$ and $\mathcal R^{ps}_6$ are irreducible, it follows that these spaces are birational. \par
Let $(C, \eta, s)$ be general and $Q \in \vert \mathcal O_{\mathbb P}(2) \vert$ such that the discriminant of $u \vert Q$ is $(C, \eta)$ and $s$ is the Steiner map. Consider $Q' = \phi'(Q)$ then the discriminant of $u' \vert Q'$  is a genus $6$ Prym sextic $(C', \eta')$ such that the following diagram is commutative:
$$
\begin{CD}
{s'(C}) @<{\phi^{-1} }<< {s(C)} \\
@A{s'}AA @AsAA \\
{C'} @<{\sigma}<< C. \end{CD}
$$
So we have $C' = \sigma_*C$ and $\sigma^* \eta' \cong \eta$. The construction of $(C', \eta')$ from $(C, \eta, s)$ defines a rational map 
\begin{equation} \label{ALFA} \alpha: \widetilde{\mathcal R}^{cb}_6 \to \mathcal R^{ps}_6. \end{equation}
The inverse construction is clear: from $(C', \eta')$ one retrieves the sequence (\ref{exact sequence}) and hence $Q'$. The strict transform of $Q'$ by $\phi$ is $Q$ and $Q$ defines $(C, \eta, s)$. This defines $\beta = \alpha^{-1}$. \end{proof}
  The next rationality result follows immediately  from the latter theorem and (\ref{NEWRATIONAL}).
 \begin{theorem}[Theorem A] The moduli space of $4$-nodal Prym sextics is rational. \end{theorem}

\section{Moduli of Igusa pencils and of $E_6$-quartics} 
In this  section we study two new moduli spaces related to the previous ones,  the \it moduli of Igusa pencils \rm and the \it  moduli of $E_6$-quartic threefolds\rm. 
Both are defined as GIT quotients and we prove that both are rational. Then we discuss some rational maps between these four spaces.   \par As we know an  Igusa pencil is generated by  the Igusa quartic $B$, and by a double quadric. Therefore we can assume that its equation is 
$ \lambda a^2 + \mu b = 0$, where $b$ is the equation of $B$ and $a$ is a non zero quadratic form. Let
\begin{equation} \label{cone}
\mathbb V \subset   \mathbb P^{69} := \vert \mathcal O_{\mathbb P^4}(4) \vert
\end{equation}
be the union of the lines which are Igusa pencils as above. Then $\mathbb V$ is a cone, of vertex the element $B$, over the $2$-Veronese embedding of $\mathbb P^{14} := \vert \mathcal O_{\mathbb P^4}(2) \vert$. The latter is just
\begin{equation}
\mathbb V_2 := \lbrace 2A, \ A  \ \in \ \mathbb P^{14} \rbrace \subset \mathbb V.
\end{equation}
Notice also that each Igusa pencil $P$ contains a unique double quadric $2A$, otherwise each element of $P$ would be union of two quadrics. Let $P = \lbrace \lambda a^2 + \mu b = 0 \rbrace$ be \it general \rm and $X$ general in $P$. Then $\Sing X$ is the transversal intersection $A \cap B$ and consists of $30$ ordinary double points, distributed in pairs on the $15$ lines whose union is $\Sing B$. 
\begin{proposition} No element of $P \setminus \lbrace B \rbrace$ is projectively equivalent to $B$. \end{proposition}
\begin{proof}  Let $t: \mathbb P \to \mathbb P^4$ be our usual tautological morphism and $\overline t: \overline{\mathbb P} \to \mathbb P^4$ its Stein factorization as in (9). Then the pencil $\overline t^*P$
is generated by $\overline B := \overline t^{-1}(B)$ and $\overline A = \overline t^{-1}(A)$. For a general $P$ the finite double cover $\overline t: \overline A \to A$ is branched on a 
surface $A \cap B$ which is singular at the finite set $\Sing B \cap A$. Hence $\overline A$ has isolated singularities.   Let $u: \overline {\mathbb P} \dasharrow \mathbb P^1$ be the rational map defined by 
$\overline t^*P$, we consider on $P$  the following open condition: 
  \begin{enumerate} \item every $\overline X \in \overline t^* P \setminus \lbrace \overline B \rbrace$ has isolated singularities, \item the ramification scheme of $u$ is reduced at $\Sing \overline B$. \end{enumerate}  
We claim that this condition is not empty. This easily implies the statement. To prove the claim a well known pencil is available. This is the unique pencil of
$\frak S_6$-invariant quartic forms: $\lbrace s_1 =  \lambda s^2_2 + \mu s_4 = 0 \rbrace$. This satisfes (1), see \cite{VdG} theorem 4.1 and \cite{CKS} theorem 3.3. The proof
of (2) is a straightforward computation. \end{proof}
\begin{proposition} Two general Igusa pencils $P_1, P_2$ are projectively equivalent iff $$\exists \  \psi \in \Aut B \ \vert \ P_2 = \psi(P_1). $$
In the same way two general $X_1, X_2 \in \mathbb V$ are projectively equivalent iff $$\exists \  \psi \in \Aut B \ \vert \ {  X_2 = \psi(X_1). }$$
 \end{proposition}
\begin{proof} Let $P_1, P_2$ be projectively equivalent then $P_2 = \psi(P_1)$ for
some $\psi \in \Aut \mathbb P^4$. By the previous proposition no element of $P_2 \setminus \lbrace B \rbrace$ is projectively equivalent to $B$. Hence it follows $\psi(B) = B$ and $\psi \in \Aut B$. The converse is obvious. In the same way let $X_1, X_2 \in \mathbb V$ be general. Then $X_i$, ($i = 1,2$),  belongs to a unique Igusa pencil contained in the cone $\mathbb V$, say $P_i = \lbrace \lambda a^2_i + \mu b = 0 \rbrace$. Moreover $\Sing X_i$ is the transversal intersection $\Sing B \cap A_i$, where $A_i = \lbrace a_i = 0 \rbrace$. We claim that $h^0(\mathcal I_{\Sing X_i}(2)) = 1$. Under this
claim assume that $\psi$ is a projective isomorphism such that $\psi(X_1) = X_2$. Then it follows $\psi (\Sing X_1) = \Sing X_2$ and $\psi(A_1) = A_2$. In particular this also implies that
$\psi(P_1) = P_2$. Hence, by the former proof, $\psi \in \Aut B$. This implies the statement up to proving the claim. For this consider the standard exact sequence of ideal sheaves
\begin{equation} \label{IDEALS}
0 \to \mathcal I_{\Sing B}(2) \to \mathcal I_{\Sing X_i}(2) \to \mathcal I_{\Sing X \vert \Sing B}(2)\to 0.
\end{equation}
Then observe that $\mathcal I_{\Sing X \vert \Sing B}(2) \cong \mathcal O_{\Sing B}$    because $\Sing X$ is a quadratic section of $\Sing B$, namely by the quadric $A$.
Moreover notice that no quadric contains $\Sing B$.   Indeed a general tangent hyperplane section $Y =  \mathbb P^3 \cap B$ is a general Kummer surface, \cite{VdG}. If $\Sing B$ is in a quadric then $Z = Y \cap \Sing B$ is in a quadric $Q$ of $\mathbb P^3$ and   $\Sing Y = Z \cup \lbrace o \rbrace$. By the $16_6$ configuration of $\Sing Y$, $o$ belongs to a trope $T$ that is a conic through $6$ nodes. Then $T \subset  Q$ and $Y \cap Q$ contains the $16$ tropes: a contradiction.   Then the claim follows passing to the associated long exact sequence    of (\ref{IDEALS}). From this, since
$h^0(\mathcal I_{\Sing B}(2)) = 0$, it follows $h^0(\mathcal O_{\Sing X}(2)) = h^0(\mathcal O_{\Sing B}) = 1$.   \end{proof}
 Now recall that $\frak S_6 \cong \Aut B \subset \Aut \mathbb P^4$ and consider the induced action  of it on $\mathbb V_2$. The latter proposition motivates the next definitions.
\begin{definition} The moduli space of Igusa pencils is the GIT quotient
\begin{equation} \mathcal P^I := \mathbb V_2 \dslash { \ \frak S_6}. \end{equation}
\end{definition}
 \begin{definition} The moduli space of $E_6$-quartic threefolds is the GIT quotient
\begin{equation}  \mathcal X := \mathbb V \dslash{ \frak S_6}. \end{equation}
\end{definition}
One can deduce the rationality of $\mathcal P^I$ from the geometry of the Segre cubic primal $B^* \subset \mathbb P^{4*}$, the dual of $B$ appearing in theorem (\ref{Segreprimal}). 
A proof can be given as follows.
\begin{theorem} \label{RATIGUSA}The moduli space $\mathcal P^I$ of Igusa pencils is rational. \end{theorem}
\begin{proof}  In $\mathbb P^5$ with coordinates $(z_1: \dots z_6)$ let $\mathbb P^4$ be the hyperplane $\lbrace \sum_{i = 1 \dots 6} z_i = 0 \rbrace$, we fix the standard
representation of $\frak S_6$ on $\mathbb P^4$, acting on $\mathbb P^4$ by permutation of the coordinates of $\mathbb P^5$.  The equation of $B$ in $\mathbb P^4$ is
$b = \sum_{i  = 1 \dots 6} z_i^4 - \frac 14(\sum_{i = 1 \dots 6} z_i^2)^2$ and $\frak S_6$ acts as $\Aut B$ on $B$, see the Introduction. As is well known the unique $\frak S_6$-invariant cubic threefold in $\mathbb P^4$ is defined by
\begin{equation}
z_1^3 + \dots + z_6^3 = 0
\end{equation}
and it is the Segre cubic primal $B^*$, up to projective equivalence.  We have to prove the rationality of $\vert \mathcal O_{\mathbb P^4}(2) \vert \dslash \frak S_6$. To this purpose we consider  $B^*$ and the standard exact sequence
  \begin{equation}
0 \to \mathcal I_{\Sing B^*}(2) \to \mathcal O_{\mathbb P^4}(2) \to \mathcal O_{\Sing B^*}(2) \to  0,
\end{equation}
$\mathcal I_{\Sing B^*}$ being the ideal sheaf of $\Sing B^*$. As is well known its long exact sequence is
$$
0 \to H^0(\mathcal I_{\Sing B^*}(2))  \to H^0(\mathcal O_{\mathbb P^4}(2)) \to  H^0(\mathcal O_{\Sing B^*}(2)) \to 0.
$$
Let $I := H^0(\mathcal I_{\Sing B^*}(2))$, since $\dim I = 5$ it follows that $\mathbb P(I)$ is the linear system of polar quadrics of $B^*$. This is the family of quadrics $Q_t$ such
that $t = (t_1: \dots: t_6) \in \mathbb P^4$ and $Q_t$ is the restriction to $\mathbb P^4$ of the polar quadric at $t$ to the cubic $ \lbrace \sum_{i = 1 \dots 6} z_i^3 = 0 \rbrace$ of $\mathbb P^5$.
In other words
$$
Q_t = \lbrace z_1+ \dots + z_6 = 0 \rbrace \cap \lbrace t_1z_1^2 + \dots + t_6z_6^2 = 0 \rbrace.
$$
Clearly $\frak S_6$ acts on $\mathbb P(I)$, moreover let $\sigma \in \frak S_6$ then $\sigma(Q_t) = Q_{\sigma(t)}$. This implies that these actions of $\frak S_6$ on $\mathbb P^4$ and $\mathbb P(I)$
are isomorphic. Hence $\mathbb P(I) \dslash \frak S_6$ is rational, since it is isomorphic to $\mathbb P^4 \dslash \frak S_6$. Indeed the latter is the weighted projective space $\mathbb P(2,3,4,5,6)$, 
 \cite{D2} p.281 and 9.1. Let us prove the rationality of $\mathcal P^I$. Of course we can  fix a direct sum decomposition $W := I\oplus J$ of the previous representation $W := H^0(\mathcal O_{\mathbb P^4}(2))$ of the group $\frak S_6$, \cite{FH} 1.5. We consider the projection map $p: \mathbb P(W) \to \mathbb P(I)$ and the blowing up $j: \mathbb P(W)' \to \mathbb P(W)$ of $\mathbb P(J)$.  Then $p$ lifts to a projective bundle $p': \mathbb P(W)' \to \mathbb P(I)$. Since the action of $\frak S_6$ on $\mathbb P(W)'$ is linear and equivariant, $p'$ descends  to a projective bundle over a non empty open set $U \subset \mathbb P(I^*) \dslash{ \frak S_6}$, say 
$$
\overline p: \mathbb P(W)'  \dslash{ \frak S_6} \to  U \subset \mathbb P(I^*) \dslash{ \frak S_6}.
$$
See \cite{MF} 7.1 and \cite{SB} 7. Since $U$ is rational we have $\vert \mathcal O_{\mathbb P^4}(2) \vert \dslash{ \frak S_6} \cong \mathbb P(W)' \dslash \frak S_6
 \cong \mathbb P^{10} \times U$.
  \end{proof}
 To {  conclude} this section let us define the  dominant rational map
\begin{equation}
f:  \mathcal X  \dasharrow \mathcal P^{I} \label{FIBRATION}.
\end{equation}
Let $X \in \mathbb V$ be general and $P$ the unique Igusa pencil containing it. Denoting by $[X]$ and $[P]$ their moduli, we set by definition
$f([X]) = [P]$. Then the fibre of $f$ at $[P]$ is the image of $P$ in $ \mathcal X $ via the moduli map.  Therefore \it $f$ is a fibration  in rational curves. \rm    
 \begin{theorem}[Theorem C] \label {THEOREM f}  The moduli space $ \mathcal X $ is rational. \end{theorem}
\begin{proof} Let $f' := f \circ \sigma$, where $\sigma:  \mathcal D_6  \to  \mathcal X $ is a birational morphism and $ \mathcal D_6 $ is smooth. Then the general fibre of $f'$  is $\mathbb P^1$. Therefore it suffices to show that $f'$  admits a rational section $s: \mathcal P^I \dasharrow  \mathcal D_6 $. This implies that
$ \mathcal X $ is birational to $\mathcal P^I \times \mathbb P^1$ and the statement follows because $\mathcal P^I$ is rational. To construct $s$  we use the exceptional divisor of the cone $\mathbb V$ blown up in its vertex $\mathsf v$. Let $\tilde \sigma: \widetilde{\mathbb V} \to \mathbb V$ be such a blow up and $\tilde E$ the exceptional divisor. Then the projection of $\mathbb V$,
from $\mathsf v$ onto $\mathbb V_2$,  lifts to a $\mathbb P^1$-bundle $\tilde p: \widetilde {\mathbb V} \to \mathbb V_2$ such that $\tilde E$ is the image of the obvious biregular section $\tilde s: \mathbb V_2 \to \widetilde{\mathbb V}$. The action of $\frak S_6$ on $\mathbb V$ lifts to $\widetilde{ \mathbb V}$ and coincides on $\tilde E$ with the standard action of $\frak S_6$ on $\mathbb V_2$. Passing to the corresponding GIT quotients, it is clear that $\tilde s$ is the pull-back of a section $s$ of $f'$.
 \end{proof}

     \section{revisiting the Prym map}
  In the next section we prove our conclusive result, stated in section 1 as theorem D,  relating the Prym map $\frak  p: \mathcal R_6 \to \mathcal A_5$ and the period map $\frak j:  \mathcal X  \to \mathcal A_5$ for the family of $E_6$-quartic threefolds. To this purpose we review in this section the beautiful picture of the Prym map $\frak p$ and Donagi's tetragonal construction, \cite{Do} and \cite{DS}. The map $\frak p$ is in fact behind most results and geometric constructions we have seen so far.       \par {  Let us recall that $\frak p: \mathcal R_6 \to \mathcal A_5$ associates to the moduli point of a smooth Prym curve $(C,\eta)$ of genus $6$ the moduli point of its Prym variety $P(C,\eta)$. This is a principally polarized abelian variety of dimension $5$. It is constructed from $(C,\eta)$ via the \'etale double covering $\pi: \tilde C \to C$, defined by the nontrivial $2$-torsion line bundle $\eta \in \Pic^0 C$. Actually $P(C,\eta)$ is the connected component of zero of the so called norm map 
$$ 
Nm: \Pic^0 \tilde C \to \Pic^0 C,
$$
sending $\mathcal O_{\tilde C}(d)$ to $\mathcal O_C(\pi_*d)$. It admits a natural principal polarization. The exceptional beauty of $\frak p$ is due to its relation to the exceptional Lie algebra $\mathsf E_6$ and to cubic surfaces.
As shown by Donagi and Smith $P$, has degree $27$, \cite{DS}. Of course this is also the degree of the forgetful map $f: \tilde {\mathcal C} \to \mathcal C$,  where $\widetilde {\mathcal C}$  is the universal Fano variety of lines  over the family $\mathcal C$ of smooth cubic surfaces in $\mathbb P^3$. The mentioned theorem of Donagi shows that the monodromy groups of $\frak p$ and of $f$ are isomorphic, \cite{Do} theorem 4.2. The monodromy group is isomorphic to the Weyl group $W(\mathsf E_6)$ of $\mathsf E_6$ and acts on the fibre of $f$ as the group preserving the incidence relation of the $27$ lines of a cubic surface.}  Therefore we can view a general fibre $\mathbb F$ of $\frak p$ as the configuration of these lines. We proceed keeping this in mind.  \par From Donagi's construction it follows that two distinct elements of $\mathbb F$ are incident lines iff they are directly associated by such a construction as follows, see \cite{Do} 2.5, 4.1.  Let $(C, \eta)$ be a general genus $6$ Prym curve and $l \in \mathbb F \subset \mathcal R_6$ its moduli point. Then
$C$ has exactly $5$ line bundles $L$ of degree $4$ and with $h^0(L) = 2$, forming the Brill-Noether locus
$$
W^1_4(C) := \lbrace L \in \Pic^4(C) \ \vert \ h^0(L) \geq 2 \rbrace.
$$
From a triple $(C, \eta, L)$ one constructs as in \cite{Do} 2.5 two new  triples $(C', \eta', L')$ and $(C'', \eta'', L'')$. They define two points $l', l'' \in \mathbb F$ so that $l, l', l''$ are coplanar lines of $\mathbb F$. One says that $l', l'', l$ form a \it triality. \rm and also that $l'$, $l''$ are \it directly associated \rm to $l$.
After labeling by the subscript $i = 1 ... 5$, we obtain, from the five triples $(C, \eta, L_i)$, exactly eleven elements of $\mathbb F$: $l, l'_i, l''_i$. These are the lines of $\mathbb F$ incident to $l$. Let us fix the notation 
\begin{equation}
\mathbb F^+ = \lbrace l'_i, l''_i \ i = 1 ... 5 \rbrace \ , \ \mathbb F^- = \lbrace n_j, j = 1 ... 16 \rbrace
\end{equation}
for the set of lines respectively intersecting or non intersecting $l$. Then $\mathbb F$ decomposes as 
\begin{equation}
\mathbb F = \lbrace l \rbrace \cup \mathbb F^+ \cup \mathbb F^-.
\end{equation}
 
Now let us also recall the bijection induced by Serre duality, namely
$$
W^1_4(C) \longleftrightarrow W^2_6(C).
$$
This is sending $L \in W^1_4(C)$ to $M := \omega_C \otimes L^{-1}$, where $M$ belongs to $W^2_6(C)$  and $$ W^2_6(C):= \lbrace M \in \Pic^6 C \ \vert \ h^0(M) \geq 3 \rbrace. $$The latter set consists of five line bundles, defining five sextic models $C' \subset \mathbb P^2$ of $C$ modulo $\Aut \mathbb P^2$. This prompts us to to give a new look at  the fibre $\mathbb F$ of $P$ in terms of Prym plane sextics. Therefore let $(C', \eta')$ be a Prym plane sextic of genus $6$ so that $\nu: C \to C'$ is the normalization and $\eta \cong \nu^* \eta'$. As already remarked in (\ref{SEXTIC}) we have the exact sequence
$$
0 \to{  \mathbb Z^4_2}  \to \Pic^0 C'_2 \stackrel {\nu^*}\to \Pic^0_2 C \to 0
$$
of $2$-torsion groups and $\eta' \in {\nu^*}^{-1}(\eta)$. Moreover the embedding $C' \subset \mathbb P^2$ is determined by the line bundle $M \cong \nu^* \mathcal O_{\mathbb P^2}(1) \in W^2_6(C)$ such that $M \cong \omega_C \otimes L^{-1}$. Finally let
\begin{equation} f: \mathcal R^{ps} \to \mathcal R_6,\end{equation}
be the rational map induced by the assignment s $(C', \eta') \to (C, \eta, M) \to (C, \eta)$. Since $\vert W^2_6(C) \vert = 5$ and $\vert {\nu^*}^{-1}(\eta) \vert = 16$ one can deduce as in \cite{FV} p.524
the next property.
 \begin{proposition} The natural map $f: \mathcal R^{ps} \to \mathcal R_6$ has degree $80$. \end{proposition}
 To continue let us fix the moduli point $l \in \mathbb F$ of $(C, \eta)$ and recall from \cite{Do} that  the set $\mathbb F^+$, of ten incident lines to $l$ distinct from  $l$, is recovered applying just once the tetragonal constructions to the five line bundles of $W^1_4(C)$. In particular, fixing an element $[C,\eta]$ in the fibre $\mathbb F$ of the Prym map $P$, one has the decomposition
\begin{equation}
\mathbb F = \lbrace [C, \eta] \rbrace \  \cup \ \lbrace {  [C^+_{ik}, \eta^+_{ik}], \ i = 1 ... 5 \ , \ k = 1,2} \rbrace \ \cup \ \lbrace [C^-_j, \eta^-_j], \ j = 1 ... 16 \rbrace = \lbrace l \rbrace \ \cup \ \mathbb F^+ \ \cup \ \mathbb F^-.
\end{equation}
Here the elements labeled by + are obtained from $(C, \eta)$ applying once the tetragonal construction to each {   $L_i \in W^1_4(C)$, $i = 1 \dots 5$.  The label $^{ -}$ corresponds} to the elements of
the set $\mathbb F^-$ of lines disjoint  from $l$. Starting from $(C, \eta)$ these are obtained after a sequence of two tetragonal constructions applied to the elements of $\mathbb F$. \par
Notice that the number $16$ does not appear by chance when counting the number of Schl\"afli double sixs containg as an element the line $l$. Indeed there exist $36$ double sixs and
$36 \times 12 = 27 \times 16$. Actually a natural bijection exists between $\mathbb F^-$ and the set of double sixes containing $l$ and $n$ as elements not in the same six. The bijection 
is constructed as follows. Let $n \in \mathbb F^-$ then a standard exercise on the configuration $\mathbb F$ of $27$ lines shows that exactly five lines of $n_1 ... n_5 \in \mathbb F^-$  
satisfy the following conditions:
$$
\vert n_i \cap n_j \vert = 0 \ , \ \vert l \cap n_k \vert = 0 \ , \ \vert n \cap n_k \vert = 1,
$$
for $1 \leq i \neq j \leq 5$ and $1 \leq k \leq 5$. Since $\vert l \cap n \vert = 0$, then $n$ defines the six of disjoint lines $\lbrace l \ n_1 ... \ n_5 \rbrace$ and hence a unique double six
\begin{equation}
\lbrace l \ n_1 ... n_5 \rbrace \ , \ \lbrace n \ m_1 ... m_5 \rbrace.
\end{equation}
We do not expand further this matter but for reconsidering diagram (\ref{MAINDIAGRAMMA}):  
$$
\begin{tikzcd}[column sep=9pc]
  {  \widetilde{\mathcal R}^{cb}_6  } \arrow{d}{\tilde d}   \arrow{r}{\alpha} &{\mathcal R^{ps}_6} \arrow{r} {p}& {\mathcal R_6}   \arrow{d}{\frak p} \\
  {\mathcal R^{cb}_6}\arrow{rd}{d}  \arrow{rru}{n} \arrow{r}{u := q \circ d}  &
  { \mathcal X } \arrow{r}{\frak j} &
  {\mathcal A_5} \\
  & {\widetilde{ \mathcal X }} \arrow{u}{q} \\
\end{tikzcd}
$$
Let us recall what are the maps in the diagram:  $\alpha$ is birational and was defined in theorem ({\ref{BIRATIONAL}). The inclusions $\Aut_{\sigma} \subset \Aut S \subset \Aut B$ correspond to the inclusions $\frak S_4 \subset \frak S_5 \subset \frak S_6$. This defines a sequence of maps of GIT quotients $  \widetilde{\mathcal R}^{cb}_6  \stackrel {\tilde d} \to \mathcal R^{cb}_6 \stackrel{d} \to  \mathcal X $. We have $\deg \tilde d = 5$, $\deg d = 6$ and $\deg u = 12$, since $q$ is the degree two map defined in (\ref{TRIANGLEDIAGRAM}).
The assignment $(C,\eta,s) \to (C,\eta)$ defines the diagonal map $n$ and $p$ is defined  by $(C', \eta') \to (C,\eta)$, where $\eta \cong \nu^*\eta'$ and $\nu: C \to C'$ is the normalization. By \cite{FV} one has $\deg p = 80$. $\frak j$ is the \it period map \rm as in (\ref{JX}) and $\deg \frak p = 27$. The commutativity of the diagram easily follows  from the given descriptions of its maps, whose known degrees are sufficient to determine the other ones. In particular  we obtain the following theorem, mentioned as B in the Introduction.
\begin{theorem}[Theorem B]The period map $\frak j$ has degree $36$ and $\deg n = 16$. \end{theorem}
   
 \section{The period map $\frak j$ and the universal set of double sixes}
Since theorem $C$ is proven we pass to theorem $D$. Let us recall from our construction in (\ref{MAINDIAGRAMMA}) that the Prym map $\frak p$ defines the variety $ \mathcal D_6 $ of pairs $(\mathsf s, \mathsf a)$ such that  $\mathsf s \subset \frak p^{-1}(\mathsf a)$ is a double six and $\mathsf a \in \mathcal A_5$ is general. Also, we have the variety $\mathcal R'$ of triples $(\mathsf l, \mathsf s, \mathsf a)$ such that $(\mathsf s, \mathsf a) \in  \mathcal D_6 $ and $\mathsf l \in \mathsf s \subset \frak p^{-1}(\mathsf a)$. This gives the commutative diagram in (\ref{DIAGRAMMATWIN}):
$$
\begin{CD}
{\mathcal R'} @>{n'}>> {\mathcal R_6} \\
@VV{d'}V @VV\frak pV \\
{ \mathcal D_6 }@>{\frak j'}>>{\mathcal A_5.} \\
\end{CD}
$$
We defined $\frak j':  \mathcal D_6  \to \mathcal A_5$   as the \it universal set of double sixes \rm of $\frak p$  and $n': \mathcal R' \to \mathcal R_6$ as the \it universal
set of pointed doubles sixes \rm of $\frak p$. On the other hand the diagram
$$
\begin{CD} 
{\mathcal R^{cb}_6} @>{n}>> {\mathcal R_6} \\
@VV{u'}V @VV\frak pV \\
{ \mathcal X }@>{\frak j}>>{\mathcal A_5}, \\
\end{CD}
$$
appears in our main commutative diagram (\ref{MAINDIAGRAMMA}) as a subdiagram. 
 \begin{theorem}[Theorem D]   The period map $\frak j:  \mathcal X  \to \mathcal A_5$ for $E_6$-quartic threefolds and the universal set $\frak j':  \mathcal D_6  \to \mathcal A_5$ of double sixes of $\frak p$ are birational over $\mathcal A_5$.
\end{theorem}
\begin{proof}[\sf Proof of the theorem] We use the main diagram and Donagi's tetragonal construction for the Prym map $P$.  Let $x \in  \mathcal X $ be the moduli point of a general
 $X$. Applying the diagram we consider the fibre of $u: \mathcal R^{cb}_6 \to  \mathcal X $ at $x$. Then $u^{-1}(x)$ consists of twelve Steiner maps modulo $\Aut S$. We can put these in the following matrix:
  \begin{equation}
 \begin{bmatrix} {\small (C_1, \eta_1, s_1) \ \ ... \ ... \ \  ( C_6, \eta_6, s_6)} \\ { \small (\overline C_1, \overline \eta_1 \overline s_1) \ \ ... \ ... \ \ (\overline C_6, \overline \eta_6, \overline s_1)} \end{bmatrix}.
 \end{equation}
Since $u = q \circ d$ we have put in the columns the two elements of a same fibre of $q$. We fix the convention that $Q_i$ and $\overline Q_i$ are the conic bundles defined by the Steiner maps $(C_i, \eta_i, s_i)$ and $(\overline C_i, \overline \eta_i, \overline s_i)$, where $(C_i, \eta_i)$ and $(\overline C_i, \overline \eta_i)$ are their discriminants. 
Let $\mathsf a = j(x)$, the moduli points in $\mathcal R_6$ of these discriminants are denoted as $l_i$ and $\overline l_i$. These are points in $P^{-1}(a)$, therefore we can think of these as points of
the configuration $\mathbb F = P^{-1}(\mathsf a)$ of $27$ lines of a smooth cubic surface.  By generic smoothness we can assume that $J \circ u$ is smooth over $\mathsf a$. Then, since $\frak j \circ u = \frak p \circ n$, it follows that
$n$ is an embedding along $u^*(x)$, in particular $n \vert u^*(x)$ is injective. {\sf Claim}: Let us consider
\begin{equation}
\mathsf s := \lbrace l_1, \overline l_1 \dots l_6 , \overline l_6 \rbrace = n(u^{-1}(x))
\end{equation}
 then $\mathsf s$ is a double six of lines in $\mathbb F$. The claim implies that there exists a rational map $\phi:  \mathcal X  \dasharrow  \mathcal D_6 $, of varieties 
over $\mathcal A_5$, defined as follows. Let $[\mathsf s] \in  \mathcal D_6 $ be the moduli point of $\mathsf s$, we set by definition $\phi(x) = [\mathsf s]$. Since $P(\mathsf s) = \mathsf a$
it follows that the period map $\frak j$ factors through $\phi$ so that $\frak j = \frak j' \circ \phi$. But then $\phi$ is birational because $\deg \frak j = \deg \frak j'$. \end{proof}
\begin{proof}[\sf Proof of the claim]    At first  we prove that $l_i$, $\overline l_i$ are skew lines of the configuration $\mathbb F$. Assume not, then $(\overline C_i, \overline \eta_i)$ is obtained from $(C_i, \eta_i)$ after one step of the tetragonal construction as in \cite{Do} 2.5, see section 8.  This is equivalent to say that exactly one $L_i$ in $W^1_4(C_i)$ exists so that the set realized from $(C_i, \eta_i, L_i)$ after the tetragonal construction is $\lbrace (C_i, \eta_i, L_i), (\overline C_i, \overline \eta_i, \overline L_i), (\overline {\overline C}_i, \overline{\overline \eta}_i, \overline{\overline L}_i) \rbrace$,
where $(C_i, \eta_i), (\overline C_i, \overline \eta_i), (\overline{\overline C}_i, \overline{\overline \eta}_i)$ are the coplanar lines $l_i, \overline l_i, \overline{\overline l}_i$ in the plane spanned by $l_i, \overline l_i$. But this implies that the assignment to $(C_i, \eta_i, s_i)$  of $(C_i, \eta_i,s_i,L_i)$ defines a rational section  of  $u: \widetilde{\mathcal R}^{cb}_6 \to \mathcal R^{cb}_6$, whose degree is $5$: against the irreducibility of $\widetilde{\mathcal R}^{cb}_6$.  To continue let us consider the following sets:  
\begin{equation}
B_i :=  \lbrace m \in \mathsf s \ \vert \ m \cap l_i \neq \emptyset \  , \ m \cap \overline l_i \neq \emptyset \rbrace 
\end{equation}
and, more interestingly for this proof, the two sets
\begin{equation}
P_i :=  \lbrace m \in \mathsf s \ \vert \ m \cap l_i \neq \emptyset \  , \ m \cap \overline l_i = \emptyset \rbrace \ , \  \overline P_i :=  \lbrace m \in \mathsf s \ \vert \ m \cap \overline l_i \neq \emptyset \ ,  \ m \cap l_i = \emptyset \rbrace.
\end{equation}
We denote their cardinalities by $b_i$ and  $p_i$, $\overline p_i$. Now the standard monodromy argument applied to $d: \widetilde{\mathcal R}^{cb}_6 \to \tilde {\mathcal Q}^{E_6}$ implies $p_i = \overline p_i$ and that $p_i$ does not depend on $i$. So we denote it by $p$. For the same reason we denote $b_i$ by $b$. We claim $b = 0$ and at first prove the statement. $b = 0$ implies $p \geq 1$. If not each element of $u^{-1}(x)$ would correspond to a line $m \in \mathsf s$  which is disjoint from any other one but $m$. Hence $\mathbb F$ would contain $12$ disjoint lines: a contradiction. Notice also  that $p \leq 5$ because $5$ is the number of lines of a smooth cubic surface intersecting $l_i$ in exactly one point and not intersecting $\overline l_i$. \par
We show that $p = 5$, which implies  that $\mathsf s$ is a double six. Indeed it is a standard property of $\mathbb F$ that $P_i \cup \overline P_i \cup \lbrace l_i, \overline l_i \rbrace$ is a double six if $p = 5$: precisely it is the unique double six containing $\lbrace l_i, \overline l_i \rbrace$ as a subset intersecting both its  sixes. This implies $$ \mathsf s = P_i \cup \overline P_i \cup \lbrace l_i, \overline l_i \rbrace \ \ i = \dots 6. $$ To prove $p = 5$ we use monodromy again. Let $P_i = \lbrace n_{i, 1} \dots n_{i,p} \rbrace$, then these elements of $\mathsf s$ are lines intersecting $l_i$ and not $\overline l_i$.  Working as above, 
the property that $l_i$ and $n_{i,k}$ are incident uniquely  reconstructs a line bundle $L_{i,k} \in W^1_4(C_i)$ such that the plane spanned by $l_i$ and $n_{i,k}$ is determined by the tetragonal construction applied to $(C_i, \eta_i, L_{i,k})$.  Now recall that the degree $5$ morphism $\tilde d: \widetilde {\mathcal R}^{cb}_6 \to \mathcal R^{cb}_6$ coincides with the forgetful map. Therefore the assignment to $(C_i, \eta_i, s_i)$ of $(C_i, \eta_i, s_i, L_{ik})$ defines a \it multisection \rm of $ \tilde d $. Since $\widetilde {\mathcal R}^{cb}_6$ is irreducible and  $\deg \tilde d = 5$ it follows $p = 5$. \par
Finally we show our claim that $b = 0$. Notice that exactly $5$ lines intersect both $l_i$ and $\overline l_i$, hence $b \leq 5$. Assume $b \geq 1$ and, for $t = 1 \dots 6$, consider the
set $B_t$ of  $5$ elements of $\mathsf s$ incident to $l_t$ and $\overline l_t$. Let $\mathbb I$ be the set of $30$ pairs $(B_t, n_{t,k})$ such that $n_{t,k} \in B_t$, $k = 1 \dots 5$. Consider its 
projection $f: \mathbb I \to  \mathsf s$. Since $\vert \mathsf s \vert = 12$, the number $\ell (m)$ of sets containing a given $m \in\mathsf s$ is not constant. But then distinct values of $\ell(m)$ 
define distinct components of ${\mathcal R}^{cb}_6$: against its irreducibility.  This completes the proof.  \end{proof}

 \end{document}